\documentclass[11pt, letterpaper]{article}
\usepackage{hyperref}
\usepackage{amsfonts}
\usepackage{amssymb}
\usepackage{amsmath}
\usepackage{amsthm}
\usepackage{latexsym}
\usepackage{color}
\usepackage{amscd}
\usepackage{comment}

\newtheorem{thm}{Theorem}[section]
\newtheorem{lemma}[thm]{Lemma}
\newtheorem{prop}[thm]{Proposition}

\newtheorem{defi}[thm]{Definition}
\newtheorem{remark}[thm]{Remark}

\newtheorem*{thm*}{Theorem}
\theoremstyle{definition}

\numberwithin{equation}{section}

\newcommand{\texto}[1]{\quad\mbox{#1}\quad}

\newcommand{\calX}{\mathcal X}
\newcommand{\ma}{\mathfrak{a}}
\newcommand{\g}{\mathfrak g}

\newcommand{\hh}{\mathbb{H}}
\newcommand{\lp}{\left(}
\newcommand{\rp}{\right)}
\newcommand{\doubleint}{\int\!\!\!\!\int}

\newcommand{\dist}{\text{dist}} 

\DeclareMathOperator{\divergence}{div}
\DeclareMathOperator{\Tr}{Tr}
\DeclareMathOperator{\Ad}{Ad}

\DeclareMathOperator{\re}{Re}

\DeclareMathOperator{\grad}{grad}
\DeclareMathOperator{\trace}{Trace}

\newcommand{\be}{\begin{equation}}
\newcommand{\ee}{\end{equation}}
\newcommand{\bee}{\begin{equation*}}
\newcommand{\eee}{\end{equation*}}
\newcommand{\bea}{\begin{eqnarray}}
\newcommand{\eea}{\end{eqnarray}}
\newcommand{\bs}{\begin{split}}
\newcommand{\es}{\end{split}}

\begin{document}

 \author{Gonz\'alez, Mar\'ia del Mar\footnote{Universidad Aut\'onoma de Madrid, Campus de Cantoblanco,
Departamento de Matem\'aticas, Madrid 28049, Spain. Email: mariamar.gonzalezn@uam.es.}\\
S\'aez, Mariel
\footnote{Pontificia Universidad Cat\'olica de Chile. Santiago
Avda. Vicu\~na Mackenna 4860, Macul, 6904441 Santiago, Chile. Email: mariel@mat.uc.cl. Author partially supported by Proyecto  Fondecyt  Regular 1150014}}

\title{Fractional Laplacians and extension problems: \\the higher rank case}
\maketitle

\abstract{The aim of this paper is two-fold: first, we look at the fractional Laplacian and the conformal fractional Laplacian from the general framework of representation theory on symmetric spaces and, second, we construct new boundary operators with good conformal properties that generalize the fractional Laplacian using an extension problem in which the boundary is of codimension two.}

\section{Introduction}

The standard fractional Laplacian on $\mathbb R^n$, $(-\Delta)^\gamma$, is a pseudo-differential operator with Fourier symbol given by $|\xi|^{2\gamma}$. For $\gamma\in(0,1)$, it is well known (see \cite{Caffarelli-Silvestre}, for instance) that it can be constructed as the Dirichlet-to-Neumann operator for the following extension problem in $\mathbb R^{n+1}_+=\{(x,y)\,:\,x\in\mathbb R^{n},y>0\}$:
\begin{equation}\label{div}\left\{\begin{split}
-\divergence \lp y^{1-2\gamma} \nabla U\rp U &=0\quad \mbox{in }\mathbb R^{n+1}_+, \\
U|_{y=0}&=f\quad \mbox{on }\mathbb R^n,
\end{split}\right.\end{equation}
Indeed, we can compute
\begin{equation}\label{fractional-Laplacian}
(-\Delta_{\mathbb R^n})^\gamma f =c_{\gamma} \lim_{y\to 0}y^{1-2\gamma} \partial_y U,
\end{equation}
for the constant
${c}_{\gamma}=-\frac{2^{2\gamma-1}\Gamma(\gamma)}{\gamma\Gamma(-\gamma)}$.

But, as it was pointed out in \cite{Chang-Gonzalez} (see also the survey \cite{Gonzalez:survey}), the elliptic problem \eqref{div} is equivalent to the following: let $\mathbb H^{n+1}$ be the hyperbolic space understood as the semispace $\mathbb R^{n+1}_+$ with the metric $g^+=\frac{dy^2+|dx|^2}{y^2}$. The conformal infinity (or boundary at infinity) is given by $\{y=0\}\simeq\mathbb R^n$. Set $f$ a smooth enough function on $\mathbb R^n$, and $s\in \mathbb C$ with $Re \, s>\frac{n}{2}$, $s\not\in \frac{n}{2}+\mathbb N$. Then there is a unique solution to the scattering equation
\begin{equation}\label{equation0}-\Delta_{g^+} u-s(n-s) u=0,\quad \text{in }\mathbb H^{n+1}\end{equation}
with the asymptotic expansion near $y=0$
$$u=y^{n-s}F+y^s G,\quad\text{for some }F,G\in\mathcal C^{\infty}(\overline{\mathbb H^{n+1}}),$$
where $F$ satisfies the boundary condition $F|_{y=0}=f$. The relation to \eqref{div} comes from the conformal relation of the hyperbolic metric $g^+$ to the Euclidean one, after the changes $u=y^{n-s}U$  and $s=\frac{n}{2}+\gamma$.
Moreover, the scattering operator on $\mathbb R^n$ is defined by
\begin{equation}\label{scattering-introduction}
  \mathcal S_{\mathbb R^n}(s)f=G|_{y=0},
\end{equation}
and it can be easily shown that it is a multiple of the fractional Laplacian \eqref{fractional-Laplacian} for $s=\frac{n}{2}+\gamma$. Indeed,
$$(-\Delta)^\gamma=2^{2\gamma}\frac{\Gamma(\gamma)}{\Gamma(-\gamma)} \mathcal S_{\mathbb R^n}\big(\tfrac{n}{2}+\gamma\big).$$

Thus we see how the extension problem \eqref{div}-\eqref{fractional-Laplacian} for the construction of the fractional Laplacian can be placed into the general framework of scattering theory on hyperbolic spaces, or more generally, on asymptotically hyperbolic manifolds, which has received a lot of attention (\cite{Mazzeo-Melrose,Graham-Zworski,Joshi-SaBarreto,Baum-Juhl}, and many others). The scattering operator has the important property of {\it conformal covariance}. More precisely, this property is that if we change the background metric conformally, say from $h$ to $h_w=e^{2w}h$, then the scattering operators with respect to each metric are related by a simple \emph{intertwining} rule of the form
\begin{equation}\label{intertwining-relation}\mathcal S_{h_w}(s)f=w^{a} \mathcal S_{h}(s)(fw^b)\end{equation}
for suitable real numbers $a,b$. In particular, scattering operators in this setting are also called intertwining operators.

On the other hand,  hyperbolic space is just the first example of a symmetric space of rank one. Naively speaking, the rank of a symmetric space is the codimension of its boundary at infinity, which in the case of the upper half-space model for hyperbolic space is just $\mathbb R^n$. There is a well established theory of scattering operators on symmetric spaces of any rank in the area of representation theory. However, this theory traditionally  has not been made available to the general PDE community. The first aim of this paper is to gather all the necessary results from representation theory to construct intertwining operators as Dirichlet-to-Neumann operators for an extension problem related to \eqref{div} when the boundary has any codimension. This is the content of Sections \ref{background}, \ref{joint-eigenspaces} and \ref{section-intertwining}.

These results rely on the particular group structure of a symmetric space. Thus if one wishes to construct a scattering operator on a more general setting, such an asymptotically product manifold (defined rigorously in Section \ref{section:asymptotically-product}), then one needs to provide a new direct proof.

Thus our second aim is to look carefully at the product case, which is of rank two, but it already contains the difficulties of the general higher rank case. This provides a construction for a scattering operator as the Dirichlet-to-Neumann operator for a suitable extension problem
 in the spirit of \eqref{div} but involving two extension variables $y_1,y_2$. This operator is conformally covariant in the sense of \eqref{intertwining-relation}, but it depends on two parameters $\gamma_1,\gamma_2$. We give a new direct proof that does not depend on the general symmetric space structure, so that it is easier to adapt to the curved case (on more general asymptotically hyperbolic manifolds). This is done mostly in Section \ref{section:asymptotically-product}.

There are already several works dealing with the resolvent operator for \eqref{equation0} on higher rank symmetric spaces, and its geometric properties \cite{Mazzeo-Vasy:noncompact,Biquard-Mazzeo:parabolic}, especially in relation to quantum $n$-body particle scattering. In the specific product case, we look at the relation to the scattering operator constructed in \cite{Biquard-Mazzeo:products,Mazzeo-Vasy:products,Huang:thesis}, which does not satisfy the conformal covariance property but it has been an inspiration for our work. The difference from their construction to ours comes from the distinction between weakly and strongly harmonic functions (see, for instance, the survey \cite{Koranyi:survey}), and the difference between the Martin boundary and the geometric boundary. This  is explored in Section \ref{section:Martin-boundary}.

Some technical difficulties come from the fact that the extension problem \eqref{div} becomes a system in the higher rank case. This is because a single equation of the type \eqref{equation0} does not have a unique solution when we impose a Dirichlet data $f$ on a submanifold of codimension higher than one. The trick is to decide which of these solutions is the suitable one for our construction, which will be decided from a system of equations.

The classical theory for such systems in the particular case of a  symmetric space is very old (see \cite{KKMOOT}), and it applies to boundary problems of systems of differential equations with regular singularities. It was developed for the calculation of joint eigenvalues of differential operators on symmetric spaces in order to solve a conjecture by Helgason \cite{Helgason:duality}. Note that another recent work, but in a different context, where systems have been used to understand a combination of fractional Laplacians is \cite{Cabre-Serra}.

Symmetric spaces of rank one have been completely understood from this point of view, since their study relies on ODE theory. An concrete example in the complex case can be found in \cite{Frank-Gonzalez-Monticelli-Tan}, where the authors construct the CR fractional Laplacian on the Heisenberg group as the conformal infinity of the complex hyperbolic space. They also give the energy formulation of this construction.

In the future we hope to explore other non-product symmetric spaces of rank bigger than one and the generalization to the curved case, always from the analytic point of view. A  first approach on this direction can be found in
\cite{Mazzeo-Vasy:SL3,Mazzeo-Vasy:SL3bis} where the authors study  the resolvent operator on
$SL_3(\mathbb R)/SO(3)$. However, it remains to characterize the conformally covariant boundary operators in this case.

The organization of the paper is the following: in Section 2 we review some classical background on representation theory of symmetric spaces. We include two very elementary subsections that develop explicitly the case of hyperbolic space for illustrative purposes, and that are aimed towards the unfamiliar reader. Experts can  directly to the next section. Section 3 and 4 contain  the construction of the scattering operators in the higher rank case, including the Poisson kernel and the study of systems of differential equations with regular singularities. Our main novel contributions come in Sections 5 and 6. In Section 5 we start developing the setting for the product of two hyperbolic spaces, while in Section 6 we provide the construction of the scattering operator in some asymptotically product manifolds.

Since this paper is aimed for the PDE community, we have tried to be self-contained as much as possible. In particular, Sections 5 and 6 do not require any representation theory background.

\section{Background} \label{background}

Here we cover classical concepts of Lie groups and representation theory. Basic references on representation theory of Lie groups are, for instance, \cite{Fulton-Harris,Knapp}, and more specifically on symmetric spaces \cite{Borel:book} or the three books by Helgason \cite{Helgason:libro2,Helgason:libro,Helgason:libro3}. The unfamiliar reader may read first Subsection \ref{subsection:hyperbolic-plane}, which deals with the simplest model: the hyperbolic plane and, in particular, the half-space model.

\subsection{The Iwasawa decomposition in representation theory}

Let $\mathcal X$ be a Riemannian symmetric space of the noncompact type. That is $\mathcal X=\mathcal G/K$, where $\mathcal G$ is a  connected semisimple Lie group with finite center and $K$ is a maximal compact subgroup of $\mathcal G$. Let $\mathfrak{g}$ be the Lie algebra associated to  $\mathcal G$ and let $\mathfrak k\subset\mathfrak g$ be the Lie algebra associated to $K$. We denote by $\exp$ the usual exponential map $\exp:\mathfrak g\to \mathcal G$, and by $\log$ its inverse.

A more geometric, but equivalent definition of a symmetric space comes by demanding that, at every point $x$ in a Riemannian manifold $\mathcal X$ there exists a global isometry $i_x$ such that $i_x$ reverses the geodesics passing through $x$.

The adjoint representation on the Lie group is defined by the map $\mathcal G\to \mbox{Aut} (\mathcal G)$ that for each $g\in \mathcal G$ it associates the automorphism
$\Ad_g(h)=ghg^{-1}$. This adjoint representation induces the adjoint representation on the Lie algebra $\mathfrak g$ by considering its derivative at the identity, so for each element $X\in\mathfrak g$, one has
$$\Ad(X)(H)=[X,H], \mbox{ for all } H\in\mathfrak g.$$
We denote by $\langle\cdot,\cdot\rangle$ the Killing form on $\mathfrak g$ given by  $$\langle X,Y\rangle=\Tr(\Ad(X)\circ \Ad(Y)).$$

It is possible to prove that
there exists a Cartan involution $\theta:\mathfrak g\to\mathfrak g$ with fixed point algebra $\mathfrak k$, which means that $\theta$ is an involutive automorphism such that the bilinear form on $\mathfrak{g}$ given by $(X,Y)\mapsto -\langle X,\theta Y\rangle$ is positive definite and such that the $+1$ eigenspace of $\theta$ coincides with $\mathfrak k$. Then the decomposition of $\mathfrak g$ into $+1$ and $-1$ eigenspaces reads
$$\mathfrak g=\mathfrak k\oplus \mathfrak p,$$
where $\mathfrak p\subset \mathfrak g$ is precisely the orthogonal complement with respect to the Killing form. We give $\mathcal X$ the Riemannian structure induced by the Killing form $\langle \cdot,\cdot\rangle$ restricted to $\mathfrak p$.


Fix a maximal abelian subspace $\mathfrak a$ of $\mathfrak p$; all such subspaces have the same dimension which is called the real rank of $\mathcal X$, and will be denoted by $r$ in the rest of the paper. Then $\exp(\mathfrak a)$ is a totally geodesic flat submanifold in $\mathcal X$ which is maximal with respect to this property. Let $\ma^*$ be its dual and $\mathfrak a_{\mathbb C}^*$ the complexification of $\ma^*$.
Since $\ma$ is abelian, there is a simultaneous diagonalization for the commuting family of symmetric homomorphisms $\Ad(H)$ where $H\in\ma$ and we may define the set of restricted roots $\Lambda(\mathfrak g,\mathfrak a)$ with respect to $\mathfrak a$ as the set of nonzero linear functionals $\alpha\in\mathfrak a^*$ such that a simultaneous eigenvector $X$ satisfies
$$[X,H]=\alpha(H)X \texto{for every} H\in\ma.$$
The space of such eigenvectors $X$ associated to each $\alpha\in\Lambda(\mathfrak g,\mathfrak a)$ is the root space $\g_\alpha$, and its dimension, denoted by $m_\alpha$, is known as multiplicity of the restricted root. If $\alpha,\tilde\alpha\in \ma^*_{\mathbb C}$, let $H_\alpha\in \ma_{\mathbb C}$ be determined by \begin{equation}\label{H-alpha}
\alpha (H)=\langle H_\alpha,H\rangle,\quad \text{for all} \quad H\in\ma,
\end{equation}
 and set $\langle \alpha,\tilde\alpha\rangle=\langle H_\alpha,H_{\tilde\alpha}\rangle$. Since $\langle \cdot,\cdot\rangle$ is positive definite on $\mathfrak p$ we write $|\alpha|=\langle \alpha,\alpha\rangle^{1/2}$ for $\alpha\in\ma^*$ and $|X|=\langle X,X\rangle^{1/2}$ for $X\in \mathfrak p$.


A point $H\in\mathfrak a$ is called regular if $\alpha(H)\neq 0$ for all $\alpha\in\Lambda(\mathfrak g,\mathfrak a)$, otherwise singular. We denote the set of regular points by  $\mathfrak a'\subset\mathfrak a$. Additionally,
each root $\alpha$ defines a hyperplane
$\{\alpha(H)=0\}$. These  hyperplanes divide the space into finitely many connected components which are  called
Weyl chambers.
Fix a Weyl chamber $\ma^+$ and call a restricted root $\alpha$ positive if it has positive values on $\ma^+$. We denote by $\Lambda^+(\mathfrak g,\mathfrak a)$ the subsystem of positive restricted roots.

A root $\alpha\in\Lambda^+(\g,\ma)$ is called simple if it is not a sum of two positive roots. Let $\{\alpha_1,\ldots,\alpha_r\}$ be the set of simple roots, and $\{H_1,\ldots,H_r\}$ be its dual basis in $\mathfrak a$. Note that the walls of the Weyl chamber $\ma^+$ lie on the hyperplanes $\{\alpha_1=0\},\ldots,\{\alpha_l=0\}$ and thus
  $$\ma^+=\{H\in\ma\,: \, \alpha_1(H)>0,\ldots,\alpha_r(H)>0\}.$$
Let $\mathfrak a^*$ be ordered lexicographically with respect to this basis $\{\alpha_1,\ldots,\alpha_r\}$. We set
\begin{equation}\rho=\tfrac{1}{2} \sum_{\alpha\in \Lambda^+(\mathfrak g,\mathfrak a)} m_\alpha \alpha\in \mathfrak a^*.\label{defrho}\end{equation}

Since $\mathfrak a\subset \mathfrak p$ and $\mathfrak a^+\subset \mathfrak a'$ were chosen freely, all such choices are conjugate under the adjoint action of $K$ on $\mathfrak p$. The linear transformations of $\ma$ induced by those of members of $K$ which leave $\mathfrak a$ invariant constitute the Weyl group $W$. More precisely, let $A=\exp \ma $, $A^+=\exp{\ma^+}$ and $\overline{A^+}$ the closure of $A^+$ in $\mathcal G$.
Let $M$ be the centralizer of $A$ in $K$, $M'$ the normalizer of $A$ in $K$, and define the Weyl group to be $W:=M'/M$. The Weyl group acts as a group of linear transformations on $\ma$ and also on $\ma_{\mathbb C}^*$ by
$$(\omega\lambda)(H)=\lambda(\omega^{-1}H),\quad \mbox{for } H\in\ma,\lambda\in\ma_{\mathbb C}^*, \omega\in W.$$

The Cartan decomposition of $\mathcal G$ is
 $$\mathcal G=K\overline{A^+}K,$$
that is, each $g\in \mathcal G$ can be written as $g=k_1a^+k_2$, where $k_1,k_2\in K$ and $a^+\in \overline{A^+}$. Moreover, such  $a^+=a^+(g)$ is unique. We write $a^+(g)=\exp A^+(g)$ where $A^+(g)\in \overline{\mathfrak a^+}$.

The Iwasawa decomposition of $\mathfrak g$ is obtained as follows: let $\mathfrak n\subset \mathfrak g$ be the nilpotent Lie subalgebra
$$\mathfrak n=\sum_{\alpha\in\Lambda^+(\mathfrak g,\mathfrak a)}  \g_\alpha,$$
then
$$\mathfrak g=\mathfrak k\oplus \mathfrak a\oplus\mathfrak n.$$
Now let $A$ and $N$ be the analytic subgroups of $\mathcal G$ with Lie algebras $\mathfrak a$ and $\mathfrak n$ respectively. As a consequence, one may define the Iwasawa decompositions of $\mathcal G$ as
$$\mathcal G=K A N,\quad \mathcal G=N A K.$$
 For $g\in \mathcal G$ we denote by $H(g)\in\mathfrak a$ the unique element such that $g$ can be expressed as
\begin{equation}\label{Iwasawa1}g=k_1\exp H(g)\,n_1\end{equation}
with $k_1\in K$ and $n_1\in N$, and by $A(g)\in\ma$ the unique element such that $g=n_2\exp A(g)k_2$,
with $k_2\in K$, $n_2\in N$. Note that $A(g)=-H(g^{-1})$.\\

One may also obtain  a polar coordinate decomposition of the symmetric space $\mathcal X=\mathcal G/K$ as follows: denote by $\mathcal B$ the compact homogeneous space $K/M$. This $\mathcal B$ is known as the {\em distinguished boundary} of $\mathcal X$ or as the {\em Furstenberg boundary}.  Let $A'=\exp{\ma'}$ (where as before $\ma'$ is the regular part of $\ma$) and $\mathcal G'=KA'K$, which is an open dense subset of $\mathcal G$. Let $\mathcal X'=\mathcal G'/K$ be the regular part of $\mathcal X$.
 It is well known that $K/M\cdot A'$ is a covering of order $\omega$ of the regular set $\mathcal X'$, where $\omega$ is the order of the Weyl group. Then, the polar decomposition is
 $$\mathcal X'=\mathcal B\cdot A^+.$$
From here we see that $\mathcal B$ is understood as an edge, and $A^+$ an extension in $r$ variables, in which $\mathfrak a^+$ has a basis $\{H_1,\ldots,H_r\}$, the dual of the simple roots $\{\alpha_1,\ldots,\alpha_r\}$) defined in \eqref{H-alpha}.
Under this identification we can set up the following notation: we  parameterize $A$ considering the inverse map for $t\in\mathbb R^r_+=\{t\in\mathbb R^r : t_j>0,j=1,\ldots r\}$, given by
$$a_t:=\exp \Big( -\sum_{j=1}^r (\log t_j)H_j\Big)\in A,$$
and note that $A^+$ is the restriction to $t_j\in(0,1)$, $j=1,\ldots,r$.
We denote
\begin{equation}\label{powers}a^{\nu}=\exp \{\nu (\log(a))\},\end{equation}
for $\nu\in\mathfrak a_{\mathbb C}^*$, $a\in A$.

Now we may define a horocycle in $\mathcal X$ as any orbit $N'\cdot z$ where $z\in\mathcal X$ and $N'$ is a subgroup of $\mathcal G$ conjugate to $N$. Each horocycle is a closed submanifold of $\mathcal X$ and is orthogonal to the manifold $A\cdot o$ at $o$. 
It can be written as $\xi=kaMN$ where $b=kM\in \mathcal B$ and $a\in A$ are unique. We say the Weyl chamber $kM$ is normal to $\xi$ and call $\log a$ the composite distance from the origin $o$ to $\xi$.

Given any $b\in \mathcal B$ and $z\in \mathcal X$, there exists a unique horocycle $\xi(z,b)$ passing through the point $z$ with normal $b$. We denote by $A(z,b)\in \ma$ the composite distance from  $o$ to $\xi(z,b)$, which  is calculated as
\begin{equation}A(z,b)=-H(g^{-1}k),\label{composite distance}\end{equation}
if $z=gK\in\mathcal X$ and $b=kM\in \mathcal B$,
where $H$ is given in \eqref{Iwasawa1}.\\

It is possible to observe that $\mathcal B\simeq \mathcal G/P$, where $P=MAN$ is a minimal parabolic subgroup. Thus the boundary at infinity of the symmetric space $\mathcal X$ is precisely the Furstenberg boundary $\mathcal B$, which can be given a parabolic geometry structure inherited from the group $\mathcal G$ (see, for instance, \cite{Biquard-Mazzeo:parabolic}). Note also that if the dimension of $\mathcal B$ is $n$, then the dimension of $\mathcal X$ is $n+r$. For a very good introduction to parabolic geometries, see the book \cite{Cap-Slovak}.

\subsection{The model case of the Hyperbolic Plane} \label{subsection:hyperbolic-plane}

This subsection contains standard material and  details may be found in \cite{Sugiura} and \cite{Conrad}. The hyperbolic plane (understood as the upper half-  plane $\mathbb H^2$) may be realized as a symmetric space of rank 1 in the following way: let
$$\mathcal G=SL_2(\mathbb R)=\left\{\begin{pmatrix}a & b\\ c & d\end{pmatrix} \,:\, ad-bc=1\right\},$$
which acts transitively on $\mathbb H^2$ by
$$\begin{pmatrix}a & b\\ c & d\end{pmatrix}\cdot z=\frac{a z+b}{c z+d}, \quad z\in \mathbb H^2.$$
Consider $K$ as the subgroup that leaves $i$ invariant. More precisely,
$$K=\left\{\begin{pmatrix}\cos \theta & -\sin \theta\\ \sin\theta & \cos\theta\end{pmatrix} \,:\,\theta\in\mathbb R\right\}=SO(2).$$
Then the hyperbolic plane can be understood  as $\mathcal{X}=\mathcal G/K$.

We consider the subgroups
\begin{equation*}
A=\left\{\begin{pmatrix}r & 0\\ 0 & 1/r\end{pmatrix}\,:\, r>0\right\},\qquad
N=\left\{\begin{pmatrix}1 & x\\ 0&  1\end{pmatrix}\,:\, x\in\mathbb R\right\}.
\end{equation*}
Then the Iwasawa decomposition $\mathcal G=KAN$ of an element in  $\mathcal G$ is given by
\begin{equation*}
\begin{pmatrix}a & b\\ c & d\end{pmatrix}=\begin{pmatrix}\cos \theta & -\sin \theta\\ \sin\theta & \cos\theta\end{pmatrix}\begin{pmatrix}r & 0\\ 0 & 1/r\end{pmatrix}
\begin{pmatrix}1 & x\\ 0&  1\end{pmatrix}
\end{equation*}
where
\begin{equation*}
  r=\sqrt{a^2+c^2},\quad \cos\theta=a/r,\quad \sin \theta=\frac{c}{r},\quad x=\frac{ab+cd}{a^2+c^2}.
\end{equation*}
Note that $AN=NA$ and that
$$AN=\left\{\begin{pmatrix}y & x \\ 0 & 1/y\end{pmatrix}\,:\, y>0,\,x\in\mathbb R\right\}.$$
The hyperbolic plane is recovered as the orbit of the action over $i$
\begin{equation*}
\begin{pmatrix}\sqrt y & x/\sqrt y \\ 0 & 1/\sqrt y\end{pmatrix}\cdot i=x+i y \in\mathbb H^2,
\end{equation*}

Before we continue, let us recall some basic formulae: for each $H_1,H_2\in \mathfrak a$, a scalar product (the Killing form) is given by
$$\langle H_1,H_2 \rangle=4\trace(H_1 H_2).$$
However, since it is customary to write the metric in $\mathbb H^{2}$ as
$$g^+=\frac{dy^2+|dx|^2}{y^2},$$
one should consider instead
$$\langle H_1,H_2 \rangle=2\trace(H_1 H_2).$$
Without loss of generality, this is the normalization we will use for this model, at the expense of introducing/removing factors of 2. The Laplace-Beltrami operator with respect to this metric is
$$\Delta_{g^+}=y^2 \lp \partial_{yy}+\partial_{xx}\rp.$$

The Lie algebras associated to $G$ and $A$ are
$$\mathfrak g=\left\{\begin{pmatrix}a & b \\ c &-a\end{pmatrix}\,:\, a,b,c\in\mathbb R\right\},$$
$$\mathfrak{a}=\mathbb{R}H_0\quad \text{for}\quad H_0=\begin{pmatrix}1/2 & 0\\ 0 & -1/2\end{pmatrix}.$$
Note that $A$ is abelian, and in particular its Lie algebra may be be identified with $\mathbb R$ (this is, $\mathbb H^2$ is a symmetric space of rank one).

The roots $\alpha_{\pm}$can be easily calculated. Since $\{H_0\}$ is a basis for $\mathfrak a$ with $\langle H_0,H_0\rangle=1$, it is enough to define
$$\alpha_{\pm}(H_0)=\pm 1, \quad \alpha_{\pm}(H)=\pm\langle H,H_0\rangle,\, \text{ for each }H\in \mathfrak a,$$
so there are only two Weil chambers, related by a reflection $\omega$ across the origin.  In this case, the Weyl group consists of two elements $W=\{e,\omega\}$. Moreover,
$$\rho (H_0)=\tfrac{1}{2} \,\alpha_+ (H_0)=\tfrac{1}{2}.$$
In this case, for $y>0$, we can parametrize $A$ by elements
$$a_y=\begin{pmatrix}
1/\sqrt{y} & 0 \\
0 & \sqrt{y}
\end{pmatrix}.$$
Now for every $s\in \mathbb C$ we can define
$\lambda\in \mathfrak a_{\mathbb C}^*$, $\lambda=\lambda_s$, by
$$\lambda(H_0)=s-\tfrac{1}{2}:=\gamma.$$

Here we present the Poincar\'e disk model $\mathbb D=\{z\in\mathbb C \,:\, |z|<1\}$ for the hyperbolic plane.
Let
$$\mathcal G=SU(1,1)=\left\{\left(\begin{array}{cc}\alpha & \beta\\ \bar \beta & \bar \alpha\end{array}\right) : \alpha,\beta\in\mathbb C, |\alpha^2|-|\beta|^2=1\right\},$$
which acts transitively on $\mathbb DD$ by
$$\left(\begin{array}{cc}\alpha & \beta\\ \bar \beta & \bar \alpha\end{array}\right)\cdot z=\frac{\alpha z+\beta}{\bar{\beta} z+\bar\alpha}, \quad z\in \mathbb D.$$
Consider $K$ as the subgroup that leaves the origin $o$ invariant. More precisely,
$$K=\left\{\left(\begin{array}{cc}\alpha & 0\\ 0 & \bar \alpha\end{array}\right):|\alpha|=1\right\}=SO(2).$$
Then the hyperbolic disk can be realized as $\mathcal G/K$, with the metric
$g^+=\frac{4|dz|^2}{(1-|z|^2)^2}.$
We denote
\begin{eqnarray*}
&A=\left\{\left(\begin{array}{cc}\cosh \frac{\tau}{2} & \sinh  \frac{\tau}{2}\\ \sinh  \frac{\tau}{2} & \cosh  \frac{\tau}{2}\end{array}\right): \tau\in\mathbb R\right\},\\
&N=\left\{\left(\begin{array}{cc}1+i \frac{\sigma}{2} & -i\frac{\sigma}{2}\\ i\frac{\sigma}{2} & 1-i\frac{\sigma}{2}\end{array}\right):\sigma \in\mathbb R\right\}.
\end{eqnarray*}
Then the Iwasawa decomposition of an element in  $\mathcal G$ is given by
$$\left(\begin{array}{cc}\alpha & \beta\\ \bar \beta & \bar \alpha\end{array}\right)=\left(\begin{array}{cc}\frac{\alpha+\beta}{|\alpha+\beta|} & 0\\ 0 &\frac{\bar{\alpha}+\bar{\beta}}{|\alpha+\beta|}\end{array}\right)
\left(\begin{array}{cc}\cosh \frac{\tau}{2} & \sinh  \frac{\tau}{2}\\ \sinh  \frac{\tau}{2} & \cosh  \frac{\tau}{2}\end{array}\right)\left(\begin{array}{cc}1+i \frac{\sigma}{2} & -i\frac{\sigma}{2}\\ i\frac{\sigma}{2} & 1-i\frac{\sigma}{2}\end{array}\right),$$
where
$e^{t}=|\alpha+\beta|$ and $\sigma=4\frac{\textrm{Im }\alpha\bar{\beta} }{|\alpha+\beta|}.$
The Iwasawa decomposition allows us to decompose any isometry of $\mathbb D$ as the product of at most three
elements in the groups $K$, $A$ and $N$. Then, it is possible to express each point $z\in \mathbb ED$ as an image of the
origin $o$ by some combination of elements in $K$, $A$ and $N$.

Now we compute the tangent spaces of these subgroups at the origin and denote them by $\mathfrak{g}$, $\mathfrak{k}$, $\mathfrak{a}$
and $\mathfrak{n}$ respectively. A direct computation implies that
$$\mathfrak{a}=\mathbb{R} \left(\begin{array}{cc}0 & 1\\ 1 & 0\end{array}\right).$$
As before, there are two roots $\{\alpha_{\pm}\}$. With some abuse of notation, we write
$\rho=1/2$. Now for every $s\in \mathbb C$ we can define
$\lambda\in \mathfrak a_{\mathbb C}^*$, $\lambda=\lambda_s$, by
$$ \lambda  \left(\left(\begin{smallmatrix}0 & 1\\ 1 & 0\end{smallmatrix}\right)\right)=s-\tfrac{1}{2}=:\gamma,$$
so for $\mathbb H^2$ many times we will identify $\lambda\in\mathfrak a_{\mathbb C}^*$ and $\gamma\in\mathbb R$.

In this model the hyperbolic distance has a very simple interpretation. Indeed,
$$\textrm{dist}\left( a_\tau, o\right)=\tau,$$ where
$$a_\tau=\left(\begin{array}{cc}\cosh \frac{\tau}{2} & \sinh  \frac{\tau}{2}\\ \sinh  \frac{\tau}{2} & \cosh  \frac{\tau}{2}\end{array}\right).$$

\subsection{Poisson kernels for  Hyperbolic Space and the scattering operator}\label{subsection:hyperbolic-space}

The objective of this section is to illustrate some of the key ideas in this paper for a simple model before getting into technicalities. References are, for example, the books \cite{Helgason:libro3} and \cite{Baum-Juhl}, but there are many others.

We start with the half-plane model for $\mathbb H^2$. Consider the scattering equation
$$-\Delta_{g^+} u-s(1-s)u=0 \quad\text{in }\mathbb H^2,$$
with Dirichlet-type data $f$, and let $s=\frac{1}{2}$. As we have mentioned in the Introduction,
there exists a unique solution with expansion near $\{y=0\}$ of the form
\begin{equation}\label{equation50}u=y^{\frac{1}{2}-\gamma} F+y^{\frac{1}{2}+\gamma} G, \end{equation}
where $F|_{y=0}=f$, $F,G$ smooth up to the boundary.

In the notation we have introduced in the previous subsection,
\eqref{equation50} is a simple way of writing
$$u=a_y^{(\lambda-\rho)(H_0)} F+a_y^{(-\lambda-\rho)(H_0)} G,$$
taking into account \eqref{powers}. We will many times perform this identification without further reference.\\

Let us look at the scattering equation in a general dimension, where we give the different models for hyperbolic space for convenience of the reader. \\

\noindent\emph{Half-space model.} Take
$$\mathbb H^{n+1}=\{ (x,y) \,: x\in\mathbb R^{n}, y>0\},$$
with the metric
$$g^+=\frac{dy^2+|dx|^2}{y^2}.$$ The boundary $\partial\mathbb H^{n+1}=\{y=0\}$ is identified with $\mathbb R^{n}$ with its Euclidean metric $|dx|^2$, and it is known as the conformal infinity.

We recall that the Laplace-Beltrami operator in $\hh^{n+1}$ is given by
$$\Delta_{g^+}=\left\{y^2 \Delta_{x,y}-(n-1)y\partial_y\right\}.$$
This operator is invariant under the action of the group $\mathcal G$. Moreover, it is possible to prove that all operators that are invariant under the action of $\mathcal G$ are generated by $\Delta_{\mathbb H^{n+1}}$ (see Section \ref{Section:joint-eigenvalues}) and, in particular, the joint eigenvalues for all such operators are determined by the eigenvalues of the Laplacian:
\begin{equation}\label{joint-eigenspaces-hyperbolic}-\Delta_{g^+} u-\mu u=0\quad \text{in }\mathbb H^{n+1}.\end{equation}
  We also recall that the
 $L^2$-spectrum  of $-\Delta_{g^+}$ is $[{n^2}/{4},\infty)$, so we will restrict to $\mu< {n^2}/{4}$, and denote
 \begin{equation}\label{gamma-nu}
\mu=s(n-s),\quad s=\tfrac{n}{2}+\gamma,\quad \gamma\in\big(0,\tfrac{n}{2}\big).
\end{equation}

To motivate the construction of spherical functions $\varphi_\gamma^{(n)}$ in later sections, let us construct particular solutions $\varphi$ to \eqref{joint-eigenspaces-hyperbolic} that depend only of the variable $y$, decaying at infinity, and normalized such that  $\varphi(0)=1$.
Then, considering the ``radial part" (i.e., the part that only depends on $y$) of the Laplace operator
$$R(\Delta_{g^+})=y^2 \partial_{yy}-(n-1)y\partial_y,$$
equation \eqref{joint-eigenspaces-hyperbolic} reduces to the simple ODE
\begin{equation}\label{ODE0}y^2 \partial_{yy}\varphi-(n-1)y\partial_y\varphi+\mu\varphi=0,\end{equation}
which has the explicit solution
\begin{equation}\label{spherical-hyperbolic}
\varphi(y)=y^{\frac{n}{2}} K_\gamma(y).\end{equation}
 Here $K_\gamma$ is the usual modified Bessel function.  This $\varphi$ is known as an elementary eigenfunction.

 One of the classical issues is to characterize all the eigenfunctions for \eqref{joint-eigenspaces-hyperbolic}. Taking into account the non-radial coordinates, one calculates the Poisson kernel for this problem:
\begin{equation}
\label{Poisson-hyperbolic}
P_\gamma^{(n)}(y, x):=y^{\frac{n}{2}-\gamma}\frac{y^{2\gamma}}{(|x|^2+y^2)^{\frac{n}{2}+\gamma}}, \quad x\in \mathbb R^{n},y>0.
\end{equation}
Remark that the minimal $\mu$-harmonic functions are obtained from the Poisson kernel \eqref{Poisson-hyperbolic}. Moreover, for each boundary value $f$ on $\mathbb R^n$, the solution of the eigenvalue equation \eqref{joint-eigenspaces-hyperbolic} can be expressed as the convolution
\begin{equation}\label{solution-u}u=P_\gamma^{(n)}(y,\cdot) *_x f,\end{equation}
In particular, from the asymptotics of the Bessel function $K_\gamma$ one knows that this solution \eqref{solution-u}
has the asymptotic expansion
\begin{equation}\label{scattering-asymptotics}u=y^{\frac{n}{2}-\gamma} F+y^{\frac{n}{2}+\gamma} G, \end{equation}
where $F,G$ are smooth functions up to the boundary $\partial\mathbb H^{n+1}$ with $F|_{y=0}=f$.

The scattering operator on $\mathbb R^n=\partial \mathbb H^{n+1}$ is defined as
\begin{equation}\label{scattering-one}\mathcal S_{\mathbb R^n}^\mu f:=G|_{y=0},\end{equation}
where $\mu$ and $\gamma$ are related by \eqref{gamma-nu}, and it can be understood as a Dirichlet-to-Neumann operator. Note that in the introduction we have used a different notation for the scattering operator \eqref{scattering-introduction}, but this new one is more suitable for the extension to the higher rank case.

From the asymptotics of \eqref{scattering-asymptotics} one can show that $\mathcal S^\mu_{\mathbb R^n}$ is a pseudo-differential operator which coincides with the fractional Laplacian $(-\Delta_{\mathbb R^n})^\gamma$. But its most important property is its conformal invariance: under a change of metric $\tilde h=e^{2w} h$ on $\mathbb R^n$, one has
$$\mathcal S_{\tilde h}^\mu(\cdot)=e^{-\frac{n+2\gamma}{2}w} \mathcal S_{h}^\mu(e^{\frac{n-2\gamma}{2} w}\,\cdot\,).$$

It is explained in \cite{Chang-Gonzalez,Gonzalez:survey} that problem \eqref{joint-eigenspaces-hyperbolic}-\eqref{scattering-asymptotics} is equivalent to the classical extension problem of \cite{Caffarelli-Silvestre} given by \eqref{div}-\eqref{fractional-Laplacian}. The equivalence follows from the relation $u=Uy^{\frac{n}{2}-\gamma}$; the underlying idea is to shift the point of view from the hyperbolic metric on the half-space $\mathbb R^{n+1}_+$ to the Euclidean one, since they are related by a simple conformal change. Thus there is an intrinsic link between fractional Laplacian operators and  representation theory of symmetric spaces of rank one. Moreover, this link is based in the study of the radial part of the scattering equation \eqref{joint-eigenspaces-hyperbolic}, which in this particular case reduces to \eqref{ODE0}.\\

\noindent\emph{Poincar\'e ball model}. As in the two-dimensional case, set
$$\mathbb H^{n+1}=\{ z\in \mathbb R^{n+1} \, : \, |z|< 1\}, \hbox{ with }n\geq 1,$$
with the Riemannian metric given by
$$g^+=\frac{4}{(1-|z|^2)^2}|dz|^2.$$
The boundary (or conformal infinity) $\partial \mathbb H^{n+1}$ is identified with the unit sphere $\mathbb S^{n}$.

Horocycles $\xi$ are are curves whose normal or perpendicular geodesics all converge asymptotically in the same direction. In the Poincar\'e ball model, these are circles tangent to the boundary $\mathbb S^n$. Let $b\in\mathbb S^n$ and consider a horocycle $\xi$ tangent to $\mathbb S^n$ at $b$. If $z$  is a point in $\xi$, we put $\langle z,b\rangle$:=distance from $o$ to $\xi$ (with sign). This is a non-Euclidean inner product.

Now, for $\mu\in\mathbb R$, consider the plane wave
$$e_{\gamma,b}:z\mapsto e^{(\gamma+\frac{n}{2})\langle z,b\rangle},\quad z\in \mathbb H^{n+1}.$$
Then one may check that $e_{\mu,b}$ is an eigenfunction for $-\Delta_{\mathbb H^{n+1}}$ with eigenvalue $\mu=\frac{n^2}{4}-\gamma^2$,
$$-\Delta_{\mathbb H^{n+1}} e_{\gamma,b}-\left(\tfrac{n^2}{4}-\gamma^2\right)e_{\gamma,b}=0.$$

 The theory of harmonic functions in this model is well known. In particular, the usual Poisson kernel is
 $$P(z,b)=\frac{1-|z|^2}{|b-z|^2},\quad z\in \mathbb{H}^{n+1}, b\in \partial \mathbb{H}^{N}.$$
In the notation above, the Poisson kernel may be written as
$$ P(z,b)=e^{\langle z,b\rangle},$$
while the Poisson kernel for \eqref{joint-eigenspaces-hyperbolic} (or the minimal $\mu$-harmonic functions in the sense of Definition \ref{def-minimal-harmonic}) is given by
\begin{equation}\label{Poisson-ball}
P^{\frac{n}{2}+\gamma}(\cdot,b),\quad b\in\mathbb S^{n},
 \end{equation}
under convention \eqref{gamma-nu}.

Given a smooth function $f$ on $\mathbb S^{n-1}$, a  $\mu$-harmonic function on $\mathbb H^{n+1}$ with boundary data $f$ may be written as
$$u(z)=\int_{\mathbb S^n} P^{\frac{n}{2}+\gamma}(\cdot,b)\,db.$$
In the case $\mu=0$, this reduces to the usual Poisson formula in a ball.
The spherical functions  have the form
$$\varphi_\gamma^{(n)}(z)=\int_{\mathbb S^n}P^{\frac{n}{2}+\gamma}(z,b)\,db,\quad z\in\mathbb H^{n+1}.$$

The proof of these facts is relies on the study of the radial part of \eqref{joint-eigenspaces-hyperbolic}. Historically this equation has usually been considered in a different set of coordinates (in the hyperboloid model for hyperbolic space), that we present next.\\

\noindent\emph{Hyperboloid model.} $\mathbb H^{n+1}$ may be defined as the upper branch of a hyperboloid in $\mathbb R^{n+2}$ with the metric induced by the Lorentzian metric in $\mathbb R^{n+2}$ given by $-dx_0^2+dx_1^2+\ldots+dx_{n+1}^2$, i.e., $$\mathbb H^{n+1}  =\{(x_0,\ldots,x_{n+1})\in \mathbb R^{n+2}: x_0^2-x_1^2-\ldots-x_{n+1}^2=1, \; x_0>0\}.$$ Here
$\mathbb H^{n+1}=\mathcal G/K$ for $\mathcal G=SO_e(n+1,1)$, $K=SO(n+1)$.
The relation to the previous model is given by the following change of variables:
$$(z_1,\ldots,z_{n+1})=\left(\frac{x_1}{1+x_0},\ldots,\frac{x_{n+1}}{1+x_0}\right).$$
In polar coordinates, $\mathbb H^{n+1}$ may be parameterized as
$$\mathbb H^{n+1}= \{x\in \mathbb R^{n+2}: x= (\cosh t, \sinh t \,\theta), \; t\geq 0, \; \theta\in \mathbb S^{n}\},$$
with the metric $g_{\mathbb H^{n+1}}=dt^2+\sinh^2 t\, g_{\mathbb S^n},$ where $g_{\mathbb S^n}$ is the canonical metric on $\mathbb S^{n}$.
Under these definitions the Laplace-Beltrami operator is given by
\begin{equation}\label{hyperbolic-Laplacian}\Delta_{\mathbb H^{n+1}}=\partial_{tt}+n \coth t\, \partial_t+\frac{1}{\sinh^2 t}\,\Delta_{\mathbb S^{n}}.\end{equation}
By construction, the (hyperbolic) distance between a point  $x=(\cosh t, \sinh t \,\theta)$ and the origin is precisely $t$.

Note that the scattering equation \eqref{scattering-introduction}, in the radial case, reduces to
$$\partial_{tt} \varphi+n \coth t\partial_t +\big(\tfrac{n^2}{4}-\gamma\big)\varphi=0.$$
After the change of variable $\sigma=-\sinh^2 t$ it can be rewritten as the standard hypergeometric equation
$$\sigma(\sigma-1)\varphi''(\sigma)+\big[(a+b+1)\sigma-c\big]\varphi'(\sigma)+ab\varphi(\sigma)=0$$
for
$$a=\tfrac{1}{2}\left(\tfrac{n}{2}+\gamma\right),\quad b=\frac{1}{2}\left(\tfrac{n}{2}-\gamma\right),\quad c=\tfrac{n+1}{2},$$
which has the explicit solution
$$\varphi(t)=\,_2F_1(a,b,c,-\sinh^2 t).$$
(Compare to \eqref{ODE-sol}). Here $\,_2F_1$ denotes the standard Hypergeometric function.

One may also calculate the Poisson kernel, the spherical functions and the scattering operator. In particular, the scattering operator on $\mathbb S^n$  (see \cite{Branson:sharp-inequalities}, or the survey \cite{Gonzalez:survey}) is given by a constant multiple of
$$\mathcal S_{\mathbb S^n}^\mu=\frac{\Gamma\lp A_{1/2}+\gamma+\tfrac{1}{2}\rp}{\Gamma\lp A_{1/2}-\gamma+\tfrac{1}{2}\rp}, \quad\text{where}\quad
A_{1/2}=\sqrt{-\Delta_{\mathbb S^n}+\lp\tfrac{n-1}{2}\rp^2}.$$

We conclude this section with a few words of another symmetric space, which is of rank two, and that will be briefly mentioned in the following. Consider
$\mathcal X=\mathcal G/K$ for $\mathcal G=SL_3(\mathbb R)$, $K=SO(3)$. Here $\mathfrak g$ consists of 3-matrices of trace zero, $\mathfrak k$ are skew-symmetric, and $\mathfrak p$ symmetric matrices.  Let $\mathfrak a$ be the subspace of diagonal matrices of trace zero, and denote these diagonal entries by $t_i,i=1,\ldots,3$, $t_1+t_2+t_3=0$.
This model has been considered in relation to quantum $3$-body scattering,  and specific references are, for instance,
\cite{Mazzeo-Vasy:SL3,Mazzeo-Vasy:SL3bis}.

\section{Differential operators and joint eigenspaces} \label{joint-eigenspaces}

After the introduction in the previous paragraphs for the hyperbolic space model, we move on by recalling the necessary results on the symmetric space case. The main idea is to develop the general theory to handle \eqref{joint-eigenspaces-hyperbolic}.

\subsection{Spherical functions}

The results of this section are classical and go back to the work of Harish-Chandra \cite{Harish-Chandra:sphericalI,Harish-Chandra:sphericalII} (see also the two survey papers \cite{Gangolli:spherical,Helgason:jewel}). However, we refer to the more modern exposition in
Helgason's books (\cite{Helgason:libro,Helgason:libro2,Helgason:libro3}).


Using the notation in the previous section, the Killing form $\langle \cdot,\cdot\rangle$ induces Euclidean measures on $A$, its Lie algebra $\ma$ and the dual space $\ma^*$. We multiply these measures by the factor $(2\pi)^{-r/2}$ and thereby obtain invariant measures $da$, $dH$ and $d\lambda$ on $A$, $\ma$ and $\ma^*$ respectively.

A spherical function is a continuous function $\varphi\not\equiv 0$ on $\mathcal G$ satisfying the functional equation
$$\int_K \varphi(xky)\,dk=\varphi(x)\varphi(y),\quad x,y\in \mathcal G.$$
Such a function is necessarily bi-invariant under $K$. By the results of Harish-Chandra, these are completely classified:

 \begin{thm}[\cite{Harish-Chandra:sphericalI}]\label{thm:spherical}
 The spherical functions are precisely the functions of the form
\begin{equation}\label{spherical-function}\varphi_\lambda(g)=\int_K e^{(-\lambda-\rho)(H(g^{-1}k))}\,dk,\quad g\in \mathcal G,\end{equation}
where $\lambda\in \mathfrak a_{\mathbb C}^*$ is arbitrary,  $H$ is defined in \eqref{Iwasawa1} and $\rho$ in \eqref{defrho}. Moreover, two such functions are identical $\varphi_\lambda\equiv\varphi_{\tilde{\lambda}}$ if and only if $\lambda=\omega\tilde{\lambda}$ for some $\omega$ in the Weyl group $W$.
\end{thm}

By $e^{\lambda}$ for $\lambda\in\mathfrak a^*$ we mean $e^{\lambda(a)}=e^{\lambda(\log a)}$. It is customary to geometrically rewrite the spherical function \eqref{spherical-function} as follows:
\begin{equation*}
\varphi_\lambda(z)=\int_{\mathcal B} e^{(\lambda+\rho)A(z,b)}\,db,\quad z\in\mathcal X,
\end{equation*}
where $A(z,b)$ is  defined in \eqref{composite distance} and $db$ is a suitable $K-$invariant measure on $\mathcal B$ of total measure 1.

\begin{remark} Additionally, note that $\varphi_\lambda$ takes the value one at the identity and that $\varphi_\lambda$ is an eigenfunction for any differential operator $D$ in $\mathbf D(\mathcal X)$ (these will be defined in the coming subsection, see formula \eqref{joint eigenspace}).
\end{remark}

\subsection{The Poisson transform}

Let $\lambda\in\mathfrak a^*_{\mathbb C}$.  For a smooth function defined on  $\mathcal B$, its Poisson transform $\mathcal P_\lambda(f)$ on $\mathcal X$ is defined by
\begin{equation}\label{definition-Poisson}(\mathcal P_\lambda f)(z)=\int_{\mathcal B} P_\lambda(z,b) f(b)\,db,\end{equation}
where the Poisson kernel is given by
\begin{equation}\label{Poisson-kernel}
P_\lambda(z,b)=e^{(-\lambda-\rho)H(z^{-1}b)},\quad z\in\mathcal X, b\in \mathcal B.
\end{equation}
We know $\mathcal P_\lambda(f)$ is an analytic function on $\mathcal X$.

Note that, according to \eqref{spherical-function}, the symmetric functions are given in terms of the Poisson transform of the constant function one.

\subsection{Differential operators and joint eigenvalues}\label{Section:joint-eigenvalues}

This section will follow closely the exposition in \cite{Helgason:libro3}, Chapter II.

 Let $\mathbf D(\mathcal X)$ denote the algebra of differential operators on $\mathcal X=\mathcal G/K$ which are invariant under the action of $\mathcal G$.  We also consider the algebra $\mathbf D(A)$ of differential operators on $A$ which are invariant under all translations ($A$ is abelian, thus $\mathbf D(A)$ contains the differential operators on $A$ with constant coefficients), as well as the subalgebra $\mathbf D_W(A)\subset \mathbf D(A)$ of invariant operators under the action of $W$ on $A$.

The first example of an invariant operator is the Laplace-Beltrami operator $\Delta_\mathcal X$. We note here that it coincides with the Casimir operator in this setting.

Let  $V$ be a submanifold in $\mathcal X$ and let $H$ be a Lie transformation group on $V$  with a suitable transversality condition, i.e., for each $v\in V$ the tangent space at $v$ can be written as direct sum
 $$\mathcal X_v=(H\cdot v)_v\oplus V_v.$$
 Let us denote by $\bar{f}$ the restriction of a function $f$ to $V$. Then,
   for each $D\in \mathbf D(\mathcal X)$, there exists a unique differential operator $R_H(D)$ on $V$, called the \emph{radial part} of $D$ with respect to $H$, such that
 $$\overline{(Du)}=R_H(D) \bar u,$$
for each locally invariant function $u$ on an open subset of $V$. 
This radial part can be explicitly computed as
$$R_H(D)=\delta^{-1/2} D_V\circ \delta^{1/2}-\delta^{-1/2} L_V(\delta^{1/2})=D_V+\grad_V(\log\delta),$$
where $\circ$ denotes composition of differential operators and $\delta$ is the density function for $V$, which is defined as follows: for each $v\in V$, the orbit $H\cdot v$ inherits a Riemannian structure from that of $V$. The corresponding Riemannian measure must be of the form
$$d\sigma_v=\delta(v) dh, $$
for $dh$ a left-invariant Haar measure on $H$.\\

Consider the Iwasawa decomposition of $\mathcal G=KAN$. We compute the radial parts of the Laplace-Beltrami in $\mathcal X$ for the actions of both $N$ and $K$.

First, the orbits $N\cdot o$ and $A\cdot o$ are orthogonal under the product defined by the Killing form, which guarantees the transversality condition. From the previous formula we have
$$R_N(\Delta_{\mathcal X})=e^{\rho} \Delta_A \circ e^{-\rho} - \langle \rho,\rho\rangle,$$
where $\Delta_A$ denotes the Laplacian on $A\cdot o$.
In general, for any $D\in\mathbf D(\mathcal X)$ we have the classical theorem:
\begin{thm}
The mapping defined by
\begin{equation}\label{Gamma}\mathbf\Gamma:D\mapsto e^{-\rho}R_N (D)\circ e^{\rho}\end{equation}
is an isomorphism of $\mathbf D(\mathcal X)$ onto $\mathbf D_W(A)$.
\end{thm}

Note that for the Laplace-Beltrami operator, we immediately have the explicit formula $\mathbf\Gamma(\Delta_{\mathcal X})=\Delta_A-\langle\rho,\rho\rangle$.

As a consequence of the previous Theorem, the algebra $\mathbf D(\mathcal X)$ is a (commutative) polynomial ring in $r=\textrm{rank }\mathcal{X}$ algebraically independent generators $D_1,\ldots,D_r$, whose degrees $d_1,\ldots,d_r$ are canonically determined by $\mathcal G$.\\

Before we continue with the exposition, let us give some examples. In the rank one case, a generator for $\mathbf D(\mathcal X)$ is precisely the Laplace-Beltrami operator $\Delta_{\mathcal X}$. If $\mathcal X$ is the product of two rank one symmetric spaces $\mathcal X=\mathcal X_1\times\mathcal X_2$, then $\mathbf D (\mathcal X)$ is generated by $\{\Delta_{\mathcal X_1},\Delta_{\mathcal X_2}\}$.

The next nontrivial example is $\mathcal X=SL_3(\mathbb R)/SO(3)$, a symmetric space of rank two. In this case there are two generators for $\mathbf D(\mathcal X)$ (see \cite{Oshima:commuting}):
\begin{equation*}\begin{split}
D_1:&=-(\theta_1 -1)(\theta_1-\theta_2)(\theta_2-\theta_1)\\
&+2t_1^2 t_2^2(\partial_x+y\partial_z)\partial_y \partial_z+(\theta_1-1)t_2^2\partial_{yy}-(\theta_1-\theta_2-1)t_1^2t_2^2\partial_{zz}-(\theta_2-1)t_1^2(\partial_x+y\partial_z)^2,\\
D_2:&=(\theta_1 -1)^2+(\theta_1 -1)(\theta_2-1)-(\theta_2 -1)^2\\
&-t_2^2 \partial_{yy} -t_1^2 t_2^2\partial_{zz}-t_1^2(\partial_x+y\partial_z)^2.
\end{split}\end{equation*}
Here we have defined $t_1,t_2$ as the two radial variables (in $A$), while $x,y,z\in\mathcal B$. Moreover, we have denoted
$$\theta_i:=t_i\partial_{t_i},\quad i=1,2.$$
It is interesting to observe that ellipticity is lost, since $D_1$ is of third order.\\

Now we go back to the calculation of radial parts. The action of the group $K$ on $\mathcal X$ yields a radial decomposition for the regular part of $\mathcal X$.
$$\mathcal X'=K\cdot(A^+\cdot o).$$
 Recall that here $A^+$ corresponds to the exponential of a fixed Weyl chamber.

In particular, we have have the following simple expression for the radial Laplace-Beltrami operator $\Delta$ on $\mathcal X$,
\begin{equation}\label{RK}R_K(\Delta_{\mathcal X})=\Delta_A+\sum_{\alpha\in \Lambda^+(\mathfrak g,\mathfrak a)}m_\alpha (\coth \alpha)H_\alpha,\end{equation}
where we have defined the vector $H_\alpha\in\mathfrak a$ by \eqref{H-alpha}. Note that this operator is only defined in $\mathfrak a'$ (the regular part of $\mathfrak a$) since the coefficients of the first order terms present singularities along the walls of the Weyl chambers $\{\alpha_i=0\}$, $i=1,\ldots,r$. But, on the other hand, when $\alpha\to\infty$, it  converges to a constant coefficient operator with no singularity. For instance, to $\partial_{tt}+n\partial_t$ in the case of the hyperbolic Laplacian  \eqref{hyperbolic-Laplacian}.

There is a similar behavior for a general differential operator:

\begin{prop}[\cite{Helgason:libro}, Chapter II]Given any $D\in\mathbf D(\mathcal X)$, is known that
$$R_K(D)=e^{-\rho}\mathbf\Gamma(D)\circ e^\rho+D'.$$
Moreover, $e^{-\rho}\mathbf\Gamma(D)\circ e{\rho}$ is a constant coefficient operator, and $D'$ has smaller degree than  $D$. In particular,  if we denote
$D'_{tH}$  the operator $D'$ with its coefficients evaluated at $tH$ we have that
$\displaystyle{\lim_{t\to +\infty}} D'_{tH}=0$ for any $H\in\ma^+$.
\end{prop}Thus $e^{-\rho}\mathbf\Gamma(D)\circ e^{\rho}$ can be understood as the model operator for $R_K(D)$ when we are away from the walls of the Weyl chamber. \\


Let $\chi:\mathbf D(\mathcal X)\to \mathbb C$ be an algebra-homomorphism and denote by $E_\chi$ the \emph{joint eigenspaces}
$$ E_\chi=\left\{u\in \mathcal C^\infty(\mathcal X)\,: \, D u=\chi(D)u\mbox{ for every }D\in \mathbf D(\mathcal X)\right\}.$$
Using the isomorphism \eqref{Gamma}, one can parameterize the set of these joint eigenspaces as follows: recall that the spherical function $\varphi_\lambda$
is an eigenfunction for any differential operator in $\mathbf D(\mathcal X)$. More precisely, it
satisfies the differential equation
\begin{equation}\label{equation20}R_K(D)\varphi_\lambda=\Gamma(D)(\lambda)\varphi_\lambda\end{equation}
in $A^+$. For each $\lambda\in\mathfrak a_{\mathbb C}^*$, define the character $\chi_\lambda(D)=\Gamma(D)(\lambda)$, and
\begin{equation}\mathcal E_\lambda(\mathcal X)=\{u\in\mathcal C^\infty(\mathcal X)\,:\, Du=\chi_\lambda(D)u\text{ for }D\in\mathbf D(\mathcal X)\}.\label{joint eigenspace}\end{equation}
Thanks to the isomorphism \eqref{Gamma}, these constitute all the joint eigenspaces $E_{\chi_\lambda}$ of the algebra $\mathbf D(\mathcal X)$.

Moreover, functions $u$ in  $\mathcal E_\lambda(\mathcal X)$ are characterized by the functional equation
$$\int_K u(gk\cdot z)\,dk=\varphi_\lambda(z)u(g\cdot o),\quad z\in \mathcal X,g\in \mathcal G.$$
These eigenspaces constitute canonical representations for the group $\mathcal G$. The natural representation  $\pi_\lambda$ of $\mathcal G$ on $\mathcal E_\lambda(\mathcal X)$ is
\begin{equation}\label{representation}(\pi_\lambda(g) u)(z)=u(g^{-1}\cdot z).\end{equation}

Finally, we say  that $u$ is \emph{strongly $\lambda$-harmonic} if
it belongs to the joint eigenspace $\mathcal E_\lambda(\mathcal X)$, $\lambda\in\mathfrak a_{\mathbb C}^*$. Note that, in the rank 1 case, since the Laplace Beltrami operator is the generator of all invariant operators, identifying $\lambda\in\mathfrak a_{\mathbb C}^*\sim\mathbb C^*$ with $\mu\in\mathbb C$, this notion agrees with the classical definition of a $\mu$-harmonic function, i.e., a solution to
$$-\Delta_{\mathcal X}u-\mu u=0.$$
However,
 in the higher rank case, $u$ is strongly $\lambda$-harmonic if and only if it is the solution of a system of $r$ equations (not necessarily elliptic). This type of systems will be considered in Section \ref{section-intertwining}.

\subsection{The $c$-function}

Harish-Chandra's idea to construct spherical functions was to obtain one particular solution $\Phi_\lambda$ of equation \eqref{equation20} in terms of an asymptotic expansion in which the coefficients are given by a recursion relation.  Note that in the case when $D=\Delta_{\mathcal X}$, the Laplace operator, this equation reduces to
\begin{equation}\label{equation40}
R_K(\Delta_{\mathcal X}) \varphi_\lambda+(-\langle \lambda,\lambda\rangle+\langle \rho,\rho\rangle)\varphi_\lambda=0.\end{equation}
Then he showed that the functions $\{\Phi_{\omega\lambda}\,:\, \omega\in W\}$ are linearly independent, so any spherical function must be of the form
\begin{equation}\label{spherical-formula}\varphi_\lambda=\sum_{\omega\in W} \mathbf c(\omega\lambda)\Phi_{\omega\lambda}\quad \text{on } A^+\cdot o,\end{equation}
for $\lambda\in\ma_{\mathbb C}^*$ except for a set of exceptional values.
The function $\mathbf c$ is known as Harish-Chandra's $c$-function.

A product formula was obtained in \cite{Gindikin-Karpelevi}:
$$\mathbf c(\lambda)=c_0 \prod_{\alpha\in\Lambda_o^+(\mathfrak g,\mathfrak a)} \frac{2^{-\langle \lambda,\alpha_o\rangle}\Gamma(\langle \lambda,\alpha_o\rangle)}{\Gamma\lp \tfrac{1}{2}\lp \tfrac{1}{2}m_\alpha+1+\langle \lambda,\alpha_o\rangle\rp\rp \Gamma\lp \tfrac{1}{2}\lp \tfrac{1}{2}m_\alpha+m_{2\alpha}+\langle \lambda,\alpha_o\rangle\rp\rp}.$$
It is a meromorphic function on $\ma_{\mathbb C}^*$. Here $\alpha_o$ denotes the normalized root $\alpha/\langle \alpha,\alpha\rangle$, the constant $c_0$ is given by the condition $\mathbf c(-\rho)=1$ and we denote by $\Lambda_o^+(\mathfrak g,\mathfrak a)$ the set of positive indivisible roots.

The importance of the $c$-function arises when considering the Helgason-Fourier transform of symmetric spaces, since it is the measure that appears in the Plancherel theorem for this Fourier transform. We will not consider this subject here, but refer instead to Helgason's books \cite{Helgason:libro,Helgason:libro2,Helgason:libro3}.

The asymptotic behavior of the spherical functions is well known and indeed, from the work of Harish-Chandra,
for $\lambda$ satisfying $\re \lambda\in\mathfrak a^*_+$,  and any $H\in\ma^+$,
$$\mathbf c(\lambda)=\lim_{t\to+\infty} e^{(-\lambda+\rho)(tH)} \varphi_\lambda(\exp(tH)).$$
This limit exists and is independent of $H$.

\subsection{A special case: symmetric spaces of rank one}

Joint eigenspaces for symmetric spaces of rank one have been completely understood (\cite{Helgason:libro}, for instance) since the equation  in \eqref{joint eigenspace} can be reduced to an ODE, which is of Sturm-Liouville type and can be solved in terms of special functions. The original calculation can be found in Section 13 of \cite{Harish-Chandra:sphericalI}.

The symmetric spaces one needs to consider are:
\begin{itemize}
\item For $\mathbb K=\mathbb R$, $\mathcal G = SO_e(n+1, 1)$, $K = SO(n+1)$,
\item For $\mathbb K = \mathbb C $, $\mathcal G = SU(n+1, 1)$, $K = S(U(n+1) × U(1))$,
\item For  $\mathbb K = \mathbb H$, $\mathcal G = Sp(n+1, 1)$, $K = Sp(n+1) × Sp(1)$,
\item For  $\mathbb K = \mathbb O$, $\mathcal G = F4(-20)$, $K = Spin(9)$.
\end{itemize}
Here we make the obvious identifications $\mathfrak p\simeq\mathbb R^{d(n+1)}$, where $d=\dim_{\mathbb R}\mathbb K$, $\ma\simeq\mathbb R$, $\mathfrak a^+\simeq (0,\infty)$, and the roots are
\begin{equation*}\begin{split}
&\Lambda(\mathfrak g,\mathfrak a)=\{\pm\alpha\}\quad\text{if } \mathbb K=\mathbb K=\mathbb R,\\ &\Lambda(\mathfrak g,\mathfrak a)=\{\pm \alpha,\pm 2\alpha\}\quad \text{if } \mathbb K=\mathbb C,\mathbb H,\mathbb O,
\end{split}\end{equation*} with multiplicities $$m_\alpha=dn,\quad m_{2\alpha}=d-1.$$
Then
$$\rho=\tfrac{1}{2}(m_\alpha \alpha+m_{2\alpha}2\alpha)=\tfrac{dn}{2}+d-1.$$

For $p\in\mathcal X$, consider geodesic polar coordinates around $p$. Let $S_r(p)$ denote the sphere in $\mathcal X$ with center $p$ and radius $r$. Let $A(r)$ be the area of $S_r(p)$, which, by the homogeneity of $\mathcal X$, is independent of $p$. Then (\cite{Helgason:libro3}, chapter II)
$$\Delta_{\mathcal X}=\partial_{rr}+\frac{\partial_r A(r)}{A(r)}\partial_r+\Delta_S,$$
where $\Delta_S$ is the Laplace-Beltrami operator on $S_r(p)$. An explicit expression for $A(r)$ is given in the same reference.

Consider the radial part of the Laplacian from \eqref{RK}. The scattering equation in radial coordinates \eqref{equation20}, written as \eqref{equation40} has a very simple expression in the rank one case. Indeed, let us write the ODE as
$$-R_K(\Delta_{\mathcal X}) \varphi(t)-\mu \varphi(t)=0,$$
and use the notation  \eqref{gamma-nu}.
 After the change of variable $\sigma=-\sinh^2 t$, it reduces to
$$\sigma(\sigma-1)\varphi''(\sigma)+[(a+b-1)z-c]\varphi'(\sigma)+ab\,\varphi(\sigma)=0$$
for
\begin{equation*}\begin{split}
&a=\tfrac{1}{2} \big( \tfrac{1}{2}m_\alpha+m_{2\alpha}+\gamma\big),\\
&b=\tfrac{1}{2} \big( \tfrac{1}{2}m_\alpha+m_{2\alpha}-\gamma\big),\\
&c=\tfrac{1}{2}(m_\alpha+m_{2\alpha}+1).
\end{split}\end{equation*}
Thus an elementary eigenfunction can be given an explicit formula
\begin{equation}\label{ODE-sol}\varphi(t)=\,_2F_1(a,b,c,-\sinh^2 t),\end{equation}
that has a very precise asymptotic behavior near $\partial \mathcal X$.


Note that the complex and quaternionic cases have been studied in \cite{Frank-Gonzalez-Monticelli-Tan}, from a variational interpretation together in terms of energy. For differential forms, see \cite{Carron-Pedon}.

The spherical functions are joint eigenfunctions  \eqref{joint eigenspace} that are $\mathcal G$-invariant. Since $\Delta_{\mathcal X}$ commutes with $\mathcal G$, any left translate of $\varphi$ will be also an eigenfunction for the same eigenvalue $\mu$. In particular, averaging over $K$, the function
 $$\varphi_\gamma^{(n)}(\cdot)=\int_K \varphi( k \,\cdot)\,dk$$
is a $K$ invariant function of $\Delta_{\mathcal X}$ in $\mathcal G/K$. We have seen in Theorem \ref{thm:spherical} that these constitute all the spherical functions.

\section{Strong solutions and the intertwining operators}\label{section-intertwining}

References for the background material in this section are the book \cite{Schlichtkrull}, chapter 5, and \cite{KKMOOT}.

Let $u$ be a strongly $\lambda$-harmonic function on $\mathcal X$. Then $u$ satisfies the
following system of differential equations
\begin{equation}(\mathcal M_\lambda): \, Du=\chi_\lambda(D)u, \mbox{ for all }D\in\mathbf D(\mathcal X).\label{strong system}\end{equation}
Note that this is a system of $r$ equations since $\mathbf D(\mathcal X)$ is generated by $r$ operators $D_1,\ldots,D_r$.

It is well known that, given a Dirichlet data $f$ on $\mathcal B$, its Poisson transform $\mathcal P_\lambda f$, defined in \eqref{definition-Poisson}, is a strong solution for \eqref{strong system}. Here we would like to consider the reverse question.

In order to write boundary values, one needs to introduce the space of \emph{hyperfunctions} on $\mathcal B$, denoted by $B(\mathcal B)$. A detailed description lies outside the scope of this paper, so we just say that the theory of hyperfunctions  generalizes the theory of distributions and, on a compact real analytic manifold,  hyperfunctions
 agree with  analytic functionals, that is,  continuous linear functionals on the space on analytic functions on the manifold. On the other hand, they can be understood as boundary values of analytic functions and they are natural boundary conditions for solutions to \eqref{strong system}. For further details we refer the reader to the books \cite{Schlichtkrull} and \cite{Sato-Kawai-Kashiwara}.


Next, we also  introduce density bundles, which is a standard notation in the bibliography to state the intertwining property.
For $\lambda\in\mathfrak a_{\mathbb C}^*$ we denote by $B(\mathcal G/P,L_\lambda)\simeq B(\mathcal B,L_\lambda)$, under a canonical identification, the space of hyperfunctions $f$ on $\mathcal G$ that satisfy
$$f(gman)=a^{\lambda-\rho}f(g)$$
for all $g\in \mathcal G$, $m\in M$, $a\in A$, $n\in N$. Also, according to the notation in \eqref{representation}, for $g\in \mathcal G$ and $f\in B(\mathcal B,L_\lambda)$,
$$(\pi_\lambda(g)f)(x)=f(g^{-1}x)$$
defines a representation of $\mathcal G$ on $B(\mathcal B,L_\lambda)$. In the conformal case, this is just a sophisticated notation to introduce the conformal property \eqref{intertwining-relation}.

Let $An(\mathcal X,\mathcal M_\lambda)$ denote the space of real analytic solutions of $(\mathcal M_\lambda)$ on $\mathcal X$ for $\lambda\in\mathfrak a^*_{\mathbb C}$. The crucial theorem is, without being extremely rigorous in the statement (we will elaborate on it in the coming subsections though),

\begin{thm}
For generic values of $\lambda$, the Poisson transform is a bijection of $B(\mathcal B,L_\lambda)$ onto $An(\mathcal X,\mathcal M_\lambda)$, i.e.,
 any  $u\in An(\mathcal X,\mathcal M_\lambda)$  may be represented as a Poisson integral of a hyperfunction $f$ on the boundary $\mathcal B$ and, for generic eigenvalues, $f$ is the boundary value of $u$ in a suitable sense.
\end{thm}

 This was shown for the rank one case and conjectured for the higher rank in \cite{Helgason:duality}, and later proved in general in \cite{KKMOOT}.   The proof relies on the study of systems of differential equations with regular singularities and their boundary values \cite{Kashiwara-Oshima:regular-singularities} (see the next subsections for more details).

The subsequent paper \cite{Oshima-Sekiguchi} characterizes when boundary values are actually distributions, under some exponential decay assumptions for $u$.
We note that in this context a natural question is in which sense the boundary values are attained. This problem has been addressed by several authors \cite{Koriyaniboundary1, Koriyaniboundary2, KoriyaneiHelgasonboundary, Knappboundary, Michelson, Sjoren,Schlichtkrull:paper}. One of such results stated that
 when $f$ is a continuous function,  then $f$ can be realized as the transversal limit of the ratio $u/\varphi_\lambda$, where $\varphi_\lambda$ is a spherical function and the results holds true a.e. when $f\in L^\infty$ (\cite{Michelson}). Different types of convergence (weak, $L^p$, etc.) have been studied, and there is an extensive bibliography.

\subsection{Systems of differential equations with regular singularities}

Consider a system of $r$ differential equations
$$(\mathcal M): \, L_j u=0, \quad j=1,\ldots,r,$$
defined on a neighborhood $V$ of the origin in $\mathbb R^{n+r}$, where each $P_j$ is an analytic differential operator of order $\sigma_j$. We define the chamber as $$V_+=\{(x_1,\ldots,x_n,t_1,\ldots,t_r)\in V\,:\, t_1,\ldots,t_r>0\},$$ the walls  $X_j=\{(x_1,\ldots,x_n,t_1,\ldots,t_r)\in V \,: \, t_i=0\}$, $j=1,\ldots, r$, and the edge $X=\{(x_1,\ldots,x_n,0,\ldots,0)\in V\}$. Assume also that the $L_j$ commute with each other.

\begin{defi}
  The system $(\mathcal M)$ is said to have a regular singularity along the walls $\{X_1,\ldots,X_r\}$ with edge $X$ if the following conditions hold, for each $j=1,\ldots,r$:
  \begin{itemize}
  \item[i.] $L_j$ is of the form $L_j(x,t,t\partial_x,t\partial_t)$ where, for each point $(x,t)$, $P_j(x,t,v,s)$ is a polynomial in $v\in\mathbb C^{nr}$ and $s\in\mathbb C^r$, and we have defined
$t\partial_x=(t_\alpha\partial_{x_\beta})$, $\alpha=1,\ldots,r$, $\beta=1,\ldots,n$ and $t\partial_t=(t_1\partial_{t_1},\ldots, t_r\partial_{t_r})$.

\item[ii.] The degree of $a_j(x,s):=L_j(x,0,0,s)$, which we call the inditial polynomial of $L_j$, is $\sigma_j$ for all $x$. Also, let $a_j^0(x,s)$ be the homogeneous part of degree $\sigma_j$ with respect to $s$ of $a_j(x,s)$. Then the roots of the equations $a_j^0(x,s)=0$ with respect to $s$ consist of $s=0$ only, for any $x$.
  \end{itemize}
\end{defi}

In this setting, if
$$P_j(x,t,v,s)=\sum a^j_{i_1,\ldots,i_{nr},k_1,\ldots,k_r}(x,t)v_1^{i_1}\ldots v_{nr}^{i_{nr}}s_1^{k_1}\ldots,s_r^{k_r},$$
then
$$L_j=\sum a^j_{i_1,\ldots,i_{nr},k_1,\ldots,k_r}(x,t)(t_1\partial_{x_1})^{i_1}\ldots (t_r\partial_{x_n})^{i_{nr}}(t_1\partial_{t_1})^{k_1}\ldots,(t_r\partial_{t_r})^{k_r}.$$
The roots $s(x)\in\mathbb C^r$ to the system of equations
$$a_1(x,s)=\ldots=a_r(x,s)=0$$
are called the \emph{characteristic exponents} of $(\mathcal M)$.\\

The main idea in \cite{KKMOOT} is that the system $(\mathcal M_\lambda)$ from \eqref{strong system} can be put into this form.
This system has regular singularities along the Weyl chamber walls with edge $\mathcal B\times \{0\}$. Moreover,
the characteristic exponents  are given by
$$s_\omega=(s_{\omega,1},\ldots,s_{\omega,r}),\quad \text{where}\quad s_{\omega,j}=(\rho-\omega\lambda)(H_j),\,\,j=1,\ldots,r.$$
In the following, we only consider $\lambda\in\mathfrak a^*_{\mathbb C}$ satisfying that, for each $\omega\in W$ not equal to the identity, where $W$ is the Weyl chamber,  there exists $i\in\{1,\ldots,r\}$ such that
\begin{equation}\label{hypothesis-lambda}(\omega\lambda-\lambda)(H_j) \not\in \mathbb Z,\end{equation}
so that the difference of two characteristic exponents does not belong to $\mathbb Z^r$. We denote the set of such $\lambda$ by $\tilde{\mathfrak a}^*$. This is needed in order to easily construct enough linearly independent solutions of $(\mathcal M)$ in order to build the general solution for the problem (compare to \eqref{expansion-extension}). Once $u$ is in this form, one is able to define boundary values as we do in Section
\ref{subsection:intertwining}.
Note that condition \eqref{hypothesis-lambda} is equivalent to
$$\frac{2\langle \lambda,\alpha\rangle}{\langle \alpha,\alpha\rangle}\not\in\mathbb Z\quad \text{for all}\quad \alpha\in \Lambda.$$

Before we continue with the exposition, let us look at the model case of symmetric functions. We denote
$$t^{s_\omega}=t_1^{s_{\omega,1}}\ldots t_r^{s_{\omega,r}}.$$
As we have mentioned, the symmetric functions are solutions for $(\mathcal M_\lambda)$. Moreover, we can give a more precise expansion for \eqref{spherical-formula}: for each $\omega\in W$ there exists a function $h_\omega(\lambda,t)$, defined on a neighborhood of $\tilde{\mathfrak a}^*\times\{0\}$ in $\tilde{\mathfrak a}^*\times\mathbb R^r_+$, holomorphic in $\lambda$ and real analytic in $t$, and with $h_\omega(\lambda,0)=\mathbf c(\omega\lambda)$, such that
$$\varphi_\lambda(a_tK)=\sum_{\omega\in W} h_\omega(\lambda,t)\, a_t^{\omega\lambda-\rho},$$
which, with the above identification, is written as
$$\varphi_\lambda(t)=\sum_{\omega\in W} h_\omega(\lambda,t) \,t^{s_{\omega}}.$$
This expansion is also true for any solution of $(\mathcal M_\lambda)$. Indeed,

\begin{thm}[\cite{KKMOOT}] Any solution $u\in\mathcal A(\mathcal X,\mathcal M_\lambda)$ has a formal expansion
\begin{equation}\label{expansion-extension}u=\sum_{\omega\in W} Q_{\omega}(x,D_x,D_t)\phi_{\omega\lambda}(x)\,t^{s_\omega},\end{equation}
where $Q_{\omega}(x,D_x,D_t)$ is a unique pseudo-differential operator of order 0 such that its principal symbol $\sigma(Q_\mu)(x,0,s)$ is equal to the constant 1 for $s_j\neq 0$, $j=1,\ldots,r$.
\end{thm}

\subsection{Boundary values and intertwining operators}
\label{subsection:intertwining}
Recall that $An(\mathcal X,\mathcal M_\lambda)$ denotes the space of real analytic solutions of $(\mathcal M_\lambda)$ on $\mathcal X$ for $\lambda\in\mathfrak a^*_{\mathbb C}$. Then, according to \eqref{expansion-extension}, given $u\in\mathcal A(\mathcal X,\mathcal M_\lambda)$ and $\nu=\omega\lambda$ for $\omega\in W$, we define the boundary map
$$\beta_{\nu} u:=\phi_{\nu}\in B(\mathcal B).$$
Note the inverse of the Poisson transform $\mathcal P_\lambda$ is $\mathbf c_\lambda^{-1}$ times the boundary value map $\beta_\lambda$.

It is also shown in \cite{KKMOOT} that the boundary map $\beta_\nu$ is a $\mathcal G$-homomorphism. This is, the following diagram is commutative
\begin{equation*}\begin{CD}
An(\mathcal X,\mathcal M_\nu)    @>\pi(g)>>
An(\mathcal X,\mathcal M_\nu)\\
@VV\beta_\nu V        @VV\beta_\nu V\\
B(\mathcal B)  @>\pi_\nu(g)>>  B(\mathcal B).
\end{CD}\end{equation*}
This allows to define the boundary map acting on conformal densities, i.e.,
$$\beta^{(\lambda)}:=\bigoplus_{\omega\in W}\beta_{\omega\lambda}:An(\mathcal X,\mathcal M_\lambda) \to \bigoplus_{\omega\in W} B(\mathcal B,L_{\omega\lambda}).$$

Now consider, for each $\omega\in W$, the scattering map is
\begin{equation}\label{scattering}\mathcal S^{(\lambda)}:= \beta^{(\lambda)} \mathcal P_\lambda,\end{equation}
defined componentwise by
$$\mathcal S_{\omega,\lambda}:=\beta_{\omega\lambda} \mathcal P_\lambda,\quad \omega\in W.$$
These are intertwining operators between two density spaces
\begin{equation}\label{intertwining}
\mathcal S_{\omega,\lambda}:B(\mathcal B,L_\lambda) \to B(\mathcal B,L_{\omega\lambda}).
\end{equation}
Trivially, when $\omega$ is the identity element in $W$, we have that $\mathcal S_{e,\lambda}=Id$.\\

We will precisely illustrate all this for the product case in Section \ref{section:asymptotically-product}, where we will provide a simple independent proof.\\

Finally, we worry about the case of systems of PDE with regular singularities for which $\lambda$ does not satisfy \eqref{hypothesis-lambda}. This was considered in \cite{Oshima:polos}. Their idea is that, even though the expansion \eqref{expansion-extension} does not hold since the terms $t^{s_\omega}$ do not give enough linearly independent solutions, there is a related expansion where these powers are replaced by suitable functions with log terms. In any case, one can still define boundary values $\beta u$. We will deal with this issue for an explicit example in Remark \ref{remark-poles}.

\section{Weak solutions, Martin boundary and the Poisson kernel}\label{section:Martin-boundary}

In this section we define weakly $\mu$-harmonic functions and the Martin boundary, and we prove that minimal positive weak  $\mu$-harmonic functions are indeed strong. Most of the results of this section are taken from the books \cite{GJT,Anker-Orsted}, the classical papers \cite{Martin,Karpelevic}, and the survey \cite{Koranyi:survey}.


We recall that a function $u$ is strongly $\mu$-harmonic  in $\mathcal X$ if it satisfies the system \eqref{strong system}. In contrast we define:

\begin{defi} \label{def-weak-harmonic}
A function $u$ is weakly  $\mu$-harmonic if
$$-\Delta_{\mathcal X} u-\mu u=0.$$
\end{defi}

The main objective of the Martin compactification is to introduce a notion of boundary that allows us to represent weakly  $\mu$-harmonic in terms of boundary values. The construction presented below was first proposed by Martin in \cite{Martin}.

We assume in the following that $\mu_0$ is the smallest eigenvalue of $-\Delta_{\mathcal X}$ and consider
  $\mu\leq \mu_0$. Fix $x_0\in \mathcal X$ a base point. Assume additionally that $-\Delta_{\mathcal X}-\mu $ admits a Green function $G^{\mu}(x,y)$ with Dirichlet conditions at infinity. Then the \emph{Martin compactification} of $\mathcal X$ respect to $\mu$ (that we denote by $\mathcal X(\mu)$) is the only compactification $\tilde{\mathcal X}$ of $\mathcal X$ with the properties:
\begin{enumerate}
\item For each for each $x\in \mathcal X$ it is possible to extend continuously to $\tilde{\mathcal X}$ the functions $y\to \mathcal K^{\mu}(x,y):=\left\{\begin{array}{cc}\frac{G^{\mu}(x,y)}{G^{\mu}(x_0,y)}& \hbox{ if } x\ne x_0\\ 0 & \hbox{ if } x=x_0 \hbox{ and } y\ne x_0 \\ 1 &\hbox{ if }x=y=x_0\end{array}\right.$ ; \quad and

\item the extended function separates the points of the ideal boundary $\partial\mathcal X(\mu)=\tilde{ \mathcal X}\setminus \mathcal X$.
\end{enumerate}
The function $\mathcal K^\mu$ is called the Martin Kernel.

Each point $y\in \mathcal{X}$ may be uniquely identified with the function $h_y(x)=\mathcal K^{\mu}(x,y)$. Hence, we consider $\mathcal H^\mu_1(\mathcal X)$ be the set of positive functions $u$ for which
 \begin{equation}\label{eigenvalue-equation}
 -\Delta_{\mathcal{X} }u-\mu u=0
  \end{equation}
  and $u(x_0)=1$. The set  $ \mathcal H^\mu_1(\mathcal X)$ is convex and compact in the topology of uniform convergence of compact sets.

The Martin  boundary is the set of functions $h\in \mathcal H^\mu_1(\mathcal X)$ that are limit functions, i.e $h(x)=\lim_n \mathcal K^\mu(x,y_n)$ for some sequence $(y_n)\subset \mathcal X$ that converges to infinity. Under the identification above, it is possible to show that  $\tilde{\mathcal X}$ is metrizable and $\mathcal X$ is open and dense in $\tilde{\mathcal X}$.

According to Martin's work \cite{Martin}, the construction above allows us to represent weakly $\mu$ harmonic functions in term of their boundary values. However, the representation may not be unique. In order to address this issue, we need to introduce the following notion:

\begin{defi} \label{def-minimal-harmonic}
A positive solution  $u$ is minimal if for every other solution $v$ such that $0\leq v\leq u$ we have $v= C u$ for some $0\leq C\leq 1$.
\end{defi}

We remark that all extremal points of $ \mathcal H^\mu_1(\mathcal X)$  belong to the set of limit functions, but not all of them are minimal in the sense of Definition \ref{def-minimal-harmonic}.
The subset of the Martin boundary corresponding to minimal functions, called the set of minimal points, will be denoted by $\partial_e \mathcal X(\mu)$.

 The class of minimal $\mu$-harmonic functions  forms a basis for the positive $\mu$-harmonic functions in $\mathcal X$. More precisely,  every positive solution $u$ of \eqref{eigenvalue-equation} is represented by a unique positive measure $m$ of total mass $u(x_0)$ carried by the set of extremal points of $ \mathcal H^\mu_1(\mathcal X)$, denoted by $\partial_e \mathcal X(\mu)$, as
$$u(z)=\int_{\partial_e \mathcal X(\mu)}\mathcal K^\mu(z,b)\, dm,\quad z\in\mathcal X.$$

Note that the construction above does not require the structure of a symmetric space. On the other hand, in the context of symmetric spaces a
 crucial step is to identify the set $\partial_e \mathcal X(\mu)$ and relate it to the distinguished boundary $\mathcal B$. We first observe the following interesting remark.

\begin{remark}
If $\mathcal X$ is a symmetric space the Green function $G^\mu$ exists for every $x_0$ and by composing with  isometries $x_0$ can be taken as the origin $o$. In the rank 1 case, the Martin compactification coincides with the (standard)  conic compactification. In particular, for
$\mathcal X=\mathbb{H}^{n+1}$ the Martin boundary can be identified with $\mathbb S^{n}$.
\end{remark}

 To identify the minimal boundary   the  case of symmetric spaces of  rank bigger than 1, one needs to take into account the different directions $L$ inside the Weyl chamber. However,  those directions provide different limits only for $\mu<\mu_0$ and do not appear in the Martin compactification $\mathcal X(\mu_0)$. In what follows, we will restrict then to the case $\mu<\mu_0$.\\

First, a classical theorem states that weak minimal $\mu$-harmonic functions are indeed strongly harmonic, i.e., they satisfy the system $(\mathcal M_\lambda)$ for some $\lambda\in (\overline{ \mathfrak a^+})^*$ (see  \cite{Karpelevic}, for a classical reference, and the book \cite{GJT} for a more modern account):

\begin{thm}\label{thm-minimal-solutions}
For $\mu <\mu_0$, the minimal solutions to $$-\Delta_{\mathcal X} u-\mu u=0$$ are given by
\begin{equation}\label{minimal-solutions}
e^{(\rho +\lambda)A(z,b)},
\end{equation}
where $A(z,b) $ is defined by \eqref{composite distance},  $z\in \mathcal X$, $b\in\mathcal  B$, $\rho\in   \mathfrak a^*$ is given by \eqref{defrho} and  $\lambda\in (\overline{ \mathfrak a^+})^*$  is given by $\lambda(H)=c\langle L,H\rangle$ with $c=\sqrt{\mu_0-\mu}$ and the quantity $L$ represents a direction in $\overline{\mathfrak a}^+$ of length 1.
\end{thm}
We thus define the Poisson kernel as in \eqref{Poisson-kernel},
and the set of directions,  for  $c=\sqrt{\mu_0-\mu} $,
$$S_\mu=\{\lambda\in (\overline{ \mathfrak a^+})^* \,: \,\lambda(H)=c\langle L,H\rangle,\, L \text{ is a direction in }\overline{\mathfrak a}^+ \text{ of length } 1\}.$$
As a consequence, all (weakly) $\mu$-harmonic functions on $\mathcal X$ are obtained as
\begin{equation}u(z)=\doubleint_{\mathcal B\times{S_\mu}} P_\lambda(z,b) \,dm(b,\lambda),\quad z\in\mathcal X.\label{rep formula}\end{equation}

In the particular case that $\mu =0$, it is well known  \cite{Furstenberg} that bounded solutions of $-\Delta_{\mathcal X}u=0$ are in bijection with $L^\infty (m_1)$, where $m_1$ is the measure that represents the constant function 1. For a general $\mu\neq 0$ this is not true in general since the set of directions $S_\mu$ is non-trivial.\\

Following \cite{Freire}, in the product case (not necessarily a symmetric space) we are able to prove a splitting theorem  in the spirit \eqref{minimal-solutions}, but in a more general setting:

\begin{thm}\label{thm:splitting}
Let $\mathcal X = \mathcal X_1\times \mathcal X_2$ be a Riemannian product, where $\mathcal X_1$  and
$\mathcal X_2$ are complete, noncompact, with Ricci curvature bounded below. Then
the following hold:
\begin{itemize}
\item[i.] Each minimal positive $\mu$-harmonic function $u$ on $\mathcal X$ splits as a product
$$u(z_1,z_2) = u_1^{\mu_1}(z_1)u_2^{\mu_2}(z_2),$$
where $\mu_i \leq \mu_0$, $u_i^{\mu_i}$ is a minimal positive $\mu_i$-harmonic function on $\mathcal X_i$, $i=1,2$, and $\mu_1+\mu_2=\mu$.
\item[ii.] Conversely, each product as above is a minimal positive $\mu$-harmonic
function on $\mathcal X$.
\end{itemize}
\end{thm}

\begin{proof}
It is a simple modification of \cite{Freire}.
\end{proof}

\begin{remark}
In fact, following the proof of \cite{Freire} it is possible to prove that functions that have the appropriate decay at $\partial \mathcal X$ split as product. Minimal functions are a particular case of such behavior. We will elaborate further in Theorem \ref{thm:expansion2}.
\end{remark}

Let us give the explicit formulas in \eqref{minimal-solutions} for the simplest rank two symmetric space, the product of two hyperbolic spaces.

\subsection{The product of two hyperbolic spaces  $ \mathbb H^{n_1+1}\times  \mathbb H^{n_2+1}$}\label{subsection:product-hyperbolic}

As an introduction,  we consider in this section the product of two hyperbolic spaces, although the results here could be generalized to other settings. We write $N_1=n_1+1$, $N_2=n_2+1$, and set $\mathcal X=\mathbb H^{N_1}\times  \mathbb H^{N_2}$. We parameterize each hyperbolic space by the coordinates $z_i=(b_i,y_i)$, $b_i\in\partial\mathbb H^{N_i}$, $y_i=\dist(z_i,\partial\mathbb H^{N_i})\in\mathbb R_+$, $i=1,2$.

 First we survey the results of \cite{Giulini-Woess}, where the full Martin boundary is computed for  this case. The idea, as above, is that the minimal $\mu-$harmonic functions  \eqref{minimal-solutions} split as a product of minimal $\mu_i-$harmonic functions on each hyperbolic space. In particular, we have that the invariant differential operators are  generated by $\Delta_i=\Delta_{\mathbb H^{N_i}}, \, i=1,2.$

The Laplace Beltrami operator $\Delta_{\mathcal X}$ on  is given by $$\Delta_{\mathcal X}=\Delta_{1}+\Delta_{2} \quad\text{where}\quad \Delta_i=\Delta_{\mathbb H^{N_i}}, \, i=1,2.$$
The bottom of the spectrum is $\mu_0=\left(\frac{n_1}{2}\right)^2+\left(\frac{n_2}{2}\right)^2$.
For $\mu<\mu_0$ we define the set of directions
$$S_\mu=\left\{\ell=(\mu_1,\mu_2):\mu_1\leq \left(\tfrac{n_1}{2}\right)^2,\, \mu_2\leq \left(\tfrac{n_2}{2}\right)^2, \mu_1+\mu_2=\mu\right\}.$$
We will use the following notation
\begin{equation}\label{notation1}\gamma_1:=\sqrt{\lp\tfrac{n_1}{2}\rp^2-\mu_1},\quad\gamma_2:=\sqrt{\lp\tfrac{n_2}{2}\rp^2-\mu_2},\end{equation}
so $\gamma:=(\gamma_1,\gamma_2)$ has length $|\gamma|=\sqrt{\mu_0-\mu}$.
Note that, here, a direction $\ell=(\mu_1,\mu_2)$ corresponds to a unitary direction $L=\frac{\gamma}{|\gamma|}$ from Theorem \ref{thm-minimal-solutions}.

While the Furstenberg boundary reduces to $\mathcal B= \partial \mathbb H^{N_1}\cup \partial \mathbb H^{N_2}$, the Martin boundary consists of three pieces:
\begin{equation}\label{three-pieces}(\partial \mathbb H^{N_1}\times \mathbb H^{N_2})\cup (\partial \mathbb H^{N_1}\times\partial\mathbb H^{N_2}\times S_\mu)\cup
(\mathbb H^{N_1}\times \partial\mathbb H^{N_2}).\end{equation}
The first and the third piece consist of the non-minimal points while, for the second, we have:
\begin{thm}[\cite{Giulini-Woess}]
The minimal functions in $\mathcal H^+(\mu)$ are given, in terms of the Poisson kernel for hyperbolic space \eqref{Poisson-ball}, by
\begin{equation}P_{\ell}(z,b)=P^{\frac{n_1}{2}+\gamma_1}(z_1,b_1) P^{\frac{n_2}{2}+\gamma_2}(z_2,b_2),\label{product poisson} \end{equation}
with $b=(b_1,b_2)\in \partial \hh^{N_1}\times \partial \hh^{N_2} $, $z_1\in \mathbb H^{N_1}$, $z_2\in\mathbb H^{N_2}$, $\ell=(\mu_1,\mu_2).$
\end{thm}

As a consequence, we have  from \eqref{rep formula} that any weakly $\mu$-harmonic function on $\hh^{N_1}\times\hh^{N_2}$ can be represented as
\begin{equation}u(z_1,z_2)=\int_{\partial \mathbb{H}^{N_1}\times \partial \mathbb{H}^{N_2}\times S_\mu} P_{\ell}(z_1,z_2,b_1,b_2 )d\nu(b_1,b_2,\ell),\label{representation formula for u}\end{equation}
where $P_{\ell}$ is given by \eqref{product poisson}.
The measure $d\nu(b_1,b_2,\ell)$ depends on the function $u$ and we would like to relate it to the {\em boundary values } of $u$.


Note that sometimes we write $P_\gamma$ for $P_\ell$, using the conventions in \eqref{notation1}.

Now we look at strongly $(\mu_1,\mu_2)$-harmonic functions in $\mathcal X$. In this case, they may be characterized as solutions to the system
\begin{equation*}\left\{\begin{split}-\Delta_{\hh^{N_1}}u-\mu_1u=&0,\\
-\Delta_{\hh^{N_2}}u-\mu_2u=&0.\end{split}\right.\end{equation*}
Our next results state that a weakly $\mu$-harmonic function with appropriate decay must be indeed strong:

\begin{lemma}
If the measure splits as
$\nu(b_1,b_2,\ell)=f(b_1,b_2)\delta_{(\mu_1,\mu_2)}$, for some fixed $(\mu_1,\mu_2)$, $\mu_1+\mu_2=\mu$, then any weak $\mu$-harmonic function
must be a strongly $(\mu_1,\mu_2)$-harmonic function.
\end{lemma}

\begin{proof}
Trivial since the Poisson kernel \eqref{product poisson}  splits.
\end{proof}

\begin{lemma}\label{lemma:weak=strong}
Let $u$ be a weak $\mu$-harmonic function in $\hh^{N_1}\times \hh^{N_2}$.
Consider the half space model with coordinates $(x_i, y_i)\in \hh^{N_i}$ (hence $\{y_i=0\}$ defines the boundary at infinity of $\hh^{N_i}$).
Assume that $\mu_1+\mu_2=\mu$ and that
$u(x_1,y_1,x_2,y_2) \leq C y_1^{\mu_1} y_2^{\mu_2}$,
where $C$ is a fixed constant. Then $u$ is strongly $(\mu_1,\mu_2)$-harmonic.
\end{lemma}

\begin{proof}

Let $u$ be a $\mu$-harmonic function represented as \eqref{representation formula for u}.
If the growth of $u$ is bounded by the growth of $P_{\gamma}$ for a fixed $\gamma=(\gamma_1,\gamma_2)$ we have that the measure $d\nu(b_1,b_2,\ell)$ is absolutely continuous respect to the measure  $\delta_{\gamma_1}(\ell_1) \delta_{\gamma_2}(\ell_2) P_{\gamma}(\cdot, b_1,b_2) db_1db_2$, and then $u$ is a strongly $(\mu_1,\mu_2)$ harmonic  and it can be represented as

$$u(z_1,z_2)=\int_{\partial \mathbb{H}^{N_1}\times \partial \mathbb{H}^{N_2}} P_{\gamma}(z_1,z_2,b_1,b_2 )f(b_1,b_2)d b_1 db_2,$$
where $f(b_1,b_2)$ is the Radon-Nykodim derivative of
 $d\nu(b_1,b_2,\ell)$ with respect to the measure $\delta_{\gamma_1}(\ell_1) \delta_{\gamma_2}(\ell_2) db_1db_2$ and $P_{\gamma}$ is given by \eqref{product poisson}.
The function $f$ corresponds to compute the limiting behavior of $u$ up to the appropriate decay. More precisely, given the explicit decay of the kernel $P^{\frac{n_1}{2}+\gamma_1}(z_1,b_1) P^{\frac{n_2}{2}+\gamma_2}(z_2,b_2)$,
by computing $$\lim_{(z_1,z_2)\to \partial \mathbb{H}^{N_1}\times \partial \mathbb{H}^{N_2}}y_1^{-(\frac{n_1}{2}+\gamma_1)}y_2^{-(\frac{n_2}{2}+\gamma_2)}u(z_1,z_2)$$ we obtain $f$ at $\partial \mathbb{H}^{N_1}\times \partial \mathbb{H}^{N_2}$.
\end{proof}

For the half space model of $\mathbb{H}^{N_i}$, with coordinates $(x_i,y_i)$, $x_i\in\mathbb R^{n_i}$, $y_i\in\mathbb R_+$, $i=1,2$, the computations above translate as follows: let
$$\mathcal{K}_{(\gamma_1, \gamma_2)}^i(x_1,y_1,x_2,y_2)=c_{n_1,\gamma_1}c_{n_2,\gamma_2}\frac{y_1^{2\gamma_1} y_2^{2\gamma_2}}{(|x_1|^2+|y_1|^2)^{\frac{n_1+1}{2}}(|x_2|^2+|y_2|^2)^{\frac{n_2+1}{2}}},$$
where $c_{n_i,\gamma_i}$ are appropriate dimensional constants. Define

\begin{equation}U_{(\gamma_1, \gamma_2)}^f(x_1,y_1,x_2,y_2)=\int_{\mathbb{R}^{n_1} \times \mathbb{R}^{n_2}}\mathcal{K}_{(\gamma_1, \gamma_2)}(x_1-\xi_1,y_1, x_2-\xi_2, y_2)f(\xi_1,\xi_2)d\xi_1 d\xi_2.\label{U}\end{equation}
The function $U^f$ satisfies $U_{(\gamma_1, \gamma_2)}^f(x_1,y_1,x_2,y_2)\to f(x_1,x_2)$ as $y_1,y_2 \to 0$.
Moreover,
$$u(x_1,y_1,x_2,y_2)=y_1^{\frac{n_1}{2}+\gamma_1}y_2^{\frac{n_2}{2}+\gamma_2} U^f_{(\gamma_1, \gamma_2)}(x_1,y_1,x_2,y_2)$$
is the strong solution described by Lemma \ref{lemma:weak=strong}.

Notice that all  $\mu$-harmonic solutions (weak and strong) described by \eqref{rep formula}  converge to 0 as we approach the boundary (both minimal and non-minimal boundary). In particular, to obtain the
boundary value as a limit such as the one stated in  \ref{lemma:weak=strong} it is necessary (but not sufficient) an understanding of the decay towards the minimal boundary. To see that the condition is not sufficient,
 consider $(\mu_1, \mu_2)$ and $(\bar{\mu}_1, \bar{\mu}_2)$ such that $\mu_1+\mu_2=\bar{\mu}_1+ \bar{\mu}_2=\mu$. 
 Let $f, \bar{f}$ be functions defined on  $\partial \mathbb{H}^{N_1}\times \partial \mathbb{H}^{N_2}$ and
$U_{(\gamma_1, \gamma_2)}^f, U_{(\bar{\gamma}_1,\,  \bar{\gamma}_2)}^{\bar{f}}$ given by \eqref{U}.
Then
\begin{equation*}
\tilde{u}(x_1,y_1,x_2, y_2)= y_1^{\frac{n_1}{2}+\gamma_1}y_2^{\frac{n_2}{2}+\gamma_2}U_{(\gamma_1, \gamma_2)}^f(x_1,y_1,x_2, y_2) +y_1^{2\bar{\gamma_1} }y_2^{2\bar{\gamma_2}}  U_{(\bar{\gamma_1},\,  \bar{\gamma_2})}^{\bar{f}} (x_1,y_1,x_2, y_2)\label{non-minimal-example}\end{equation*}
is a  weakly $\mu$-harmonic, but for any power $(\delta_1,\delta_2)$ the function $y_1^{-\delta_1} y_2^{-\delta_2} \tilde{u}(z_1,z_2)$ does not converge uniformly to a function as $y_1,\, y_2\to 0$. In fact, if we additionally fix $\theta\in[0,\infty]$ such that $\frac{y_1}{y_2}\to \theta$ as $y_1,y_2\to 0$ and assume  without loss of generality that $\gamma_1+\gamma_2>\bar{\gamma_1}+\bar{\gamma_2}$,
then we have as $(y_1,y_2)\to 0$ and $\frac{y_1}{y_2}\to \theta$
that

$$  y_1^{-\left(\frac{n_1}{2}+\bar{\gamma_1}\right)}y_2^{-\left(\frac{n_2}{2}+\bar{\gamma_2}\right)} \tilde{u}(x_1,y_1,x_2, y_2) \to \left\{\begin{array}{cc}f(x_1,x_ 2)&\hbox{ if }\theta\ne 0 \hbox{ or } (\theta=0\hbox{ and } \gamma_1>\bar{\gamma_1}),\\
\hbox{diverges} & \hbox{ otherwise.}\end{array}\right.
$$

Furthermore, there are solutions that do not have a power decay under any assumption. Consider for instance the solution
$$\int_{(\gamma_1, \gamma_2)\in\mathbb{S}_+}y_1^{\frac{n_1}{2}+\gamma_1}y_2^{\frac{n_2}{2}+\gamma_2}d\sigma(\gamma_1, \gamma_2),$$ where  $\mathbb{S}_+=\{(\gamma_1, \gamma_2)\,: \, \gamma_1^2+\gamma_2^2= \frac{n_1^2+n_2^2}{4}-\mu, \hbox{ and  }  \gamma_1, \gamma_2\geq 0 \}$ and $\sigma$ the standard measure on this set.

Hence, in general, we can only consider the boundary values in the sense of Martin as the projections on the corresponding kernel  representing a minimal boundary point.

On the other hand, note that a strongly $\mu$-harmonic function $u$ as above satisfies at the non-minimal boundary that, as $y_1\to 0$ (resp. $y_2\to 0$),
\begin{align*}y_1^{-\left(\frac{n_1}{2}+\gamma_1\right)}u(x_1,y_1,x_2, y_2) &\to y_2^{\frac{n_2}{2}+\gamma_2} U^f_{(\gamma_1, \gamma_2)}(x_1,0,x_2,y_2),\\
\hbox{(resp. }, y_2^{-\left(\frac{n_2}{2}+\gamma_2\right)}u(x_1,y_1,x_2, y_2) &\to y_1^{\frac{n_1}{2}+\gamma_1} U^f_{(\gamma_1, \gamma_2)}(x_1,y_1,x_2,0)).\end{align*}
In particular we have that at the non-minimal boundary the extension is $\mu_j$-harmonic in $\mathbb{H}^{N_j}$.

A similar computation cannot be generally performed for weakly $\mu$-harmonic solutions.

\section{The asymptotically product case}\label{section:asymptotically-product}

Inspired by the product of two hyperbolic spaces, first we cover the product of any two  conformally compact Einstein manifolds, and then we consider a more general perturbation.  We would like to first study strongly harmonic functions and to construct the scattering map.

\subsection{Review}

Let us start by  carefully reviewing the construction of the scattering operator on the boundary of a conformally compact Einstein manifold (see \cite{Graham-Zworski,Mazzeo-Melrose,Joshi-SaBarreto}, for instance), which is the natural generalization of a (real) symmetric space of rank one to the curved case. We use the notation $N=n+1$.

Let $(\mathcal X^{N},g^+)$ be a conformally compact Einstein manifold with conformal infinity $(\mathcal B^n,[h])$ and defining function $y$. A more general setting would be to take $(\mathcal X^{N},g^+)$ only asymptotically hyperbolic but let us restrict to the conformally compact Einstein setting for simplicity. Let us assume that $g^+$ is written in normal form, i.e.,
$$g^+=\frac{dy^2+h_y}{y^2},$$
where $h_y$ is a one-parameter family of metrics on $\mathcal B$ satisfying $h_y|_{y=0}=h.$

Take $s\in\mathbb C$, $Re(s)>n/2$,  we are actually interested in the case $s=\frac{n}{2}+\gamma$, $\gamma\in \left(0,\frac{n}{2}\right)$ not an integer.

\begin{lemma}[\cite{Graham-Zworski}]\label{lemma-dim1}
Given a smooth function $f$ on $\mathcal B$, and $s\not\in\frac{n}{2}+\mathbb N$ not in the pure spectrum of $-\Delta_{g^+}$, then there exists a unique solution of
$$-\Delta_{g^+} u-s(n-s)u=0\quad\text{in }\mathcal X$$
with the asymptotic expansion near $\mathcal B$ given by
\begin{equation}\label{expansion-GZ}u=y^{n-s}F+y^{s}G, \quad F,G\in\mathcal C^{\infty} (\overline{\mathcal X}), \quad F|_{y=0}=f.\end{equation}
Moreover, $G=\mathcal R f$, where $\mathcal R$ is given explicitly in terms of derivatives of $f$ and the resolvent for $\Delta_{g^+}$ on $(\mathcal X^{N},g^+)$.

Now, the scattering operator on $\mathcal B$, defined as $$\mathcal S(s)f = G|_\mathcal B,$$
gives a meromorphic family of pseudo-differential operators in $\mathcal B$, with principal symbol  the same as $(-\Delta_h)^\gamma$ (times a multiplicative constant).

The operator  $\mathcal S(s)$ conformally covariant. Indeed, for a change of metric
 $\tilde h=e^{2w} h$, $w>0$, on $\mathbb R^n$, one has
$$\mathcal S_{\tilde h}(s)(\cdot)=e^{-\frac{n+2\gamma}{2}w} \mathcal S_{h}(s)(e^{\frac{n-2\gamma}{2} w}\,\cdot\,).$$
The values $s = n/2, n/2 +1 , n/2 + 2, \ldots$ are simple poles of finite rank, these are known as the trivial poles.
\end{lemma}

Note that $\mathcal S(s)$ may have other poles, however, for the rest of the paper we assume that we are not in those exceptional cases. Sometimes we will write $\mathcal S^\mu$ for $\mathcal S(s)$, in the notation \eqref{gamma-nu}.

\begin{proof}
The proof in \cite{Graham-Zworski} goes by  constructing $F$ as formal solution and then completing it to an exact solution using the resolvent for the operator $Lu=-\Delta_{g^+}u -s(n-s)u$. First, the construction of $F$ is inductive: let $F=\sum_{i=0}^\infty f_{i} y^{i}$, where $f_0:=f.$
Let also
$$D = -y\partial_{y y}+(2s-n-1)\partial_{y}+y\Delta_{h_y},$$
so that
$[-\Delta_{g^+} -s(n-s)]\circ y^{n-s} =y^{n-s+1} D.$
Note that
$$D(f_{i} y^i)= i (2s_1-n_1-i)f_{i} y^{i-1}+O(y^i),$$
We claim that it is possible to choose $f_i$ such that
\begin{equation*}D (F_{i})=O(y^i),\quad i=1,2,\ldots.\end{equation*}
To construct $F_{i}$ from  the previous step $F_{i-1}$ just note that
\begin{equation}\label{cuentas1}\begin{split}
D(F_{i})&=D(F_{i-1})+D(f_{i} y^{i})\\
&=O(y^{i-1})+i(2s-n-i)f_{i} y^{i-1}+O(y^i).
\end{split}\end{equation}
Thus if $2s-n-i\neq 0$, this fixes $f_{i}$. In particular, we have shown that
\begin{equation}\label{operator-p}f_i=p_{i,s} f,\end{equation}
where $p_{i,s}$ is a differential operator on $\mathcal B$.

By using Borel's lemma, we may ensure that there is a function  $F(x,y) $ that satisfies
\begin{equation}\label{equation-F}-\Delta_{g^+} (y^{n-s}F)-s(n-s)y^{n-s}F=O(y^\infty).\end{equation}
That is, the operator evaluated at $F$ vanishes at infinite order at the boundary. Note that $F$ may be explicitly computed in terms of $p_{i,s} f$.

Now we find $G$ by observing that if  we set  $ u=y^{n-s}F+E$, we have that $E$ must satisfy
$$-\Delta_{g^+} E-s(n-s)E=-(-\Delta_{g^+} (y^{n-s}F)-s(n-s)y^{n-s}F)\in L^2(\mathcal X).$$
According to \cite{Mazzeo-Melrose} $E$ can be written in term of the resolvent $\mathcal{R}$ of $D$  and has decay $y^s$. That is
$$E=\mathcal{R} F=y^sG(x,y).$$

Note that if $G$ is real analytic then we have that
$$G=\sum_{i=1}^{\infty} g_i y^i,$$
and from the proof in \cite{Graham-Zworski} one knows that
\begin{equation}\label{operator-q}g_0=\mathcal S(s) f \quad\text{and} \quad g_i=q_{i,s} g_0,\quad i=1,2,\ldots,\end{equation}
where $q_{i,s}$ is a differential operator on $\mathcal B$.
\end{proof}

\begin{remark}
In order to perform the construction above
the formal solution $F$ does not need to be constructed to vanish  to infinite order \eqref{equation-F},  it is enough for it to be in $L^2$ (i.e to vanish to a high enough order).
\end{remark}

\begin{remark}\label{polosGZ}
In the case that $s=\tfrac{n}{2}+k$ for some $k\in\mathbb N$, there still exists a unique solution $u$ but the expansion \eqref{expansion-GZ} needs to be replaced by
$$u=y^{n-s}F+y^{s}\log y \,G.$$
The scattering operator can be still be defined and a residue argument relates it to the local conformal powers of the Laplacian $(-\Delta_{h})^k$. The idea of the proof is that the coefficient in front of $f_i$ in \eqref{cuentas1} vanish for some $i$, so there is an obstruction to solve for smooth solutions and a log term needs to be introduced (\cite{Graham-Zworski}).
\end{remark}

\subsection{The product of two conformally compact Einstein manifolds}

Let $(\mathcal X_1^{N_1},g_1^+)$ and $(\mathcal X_2^{N_2},g_2^+)$ be two conformally compact Einstein manifolds  with conformal infinities $(\mathcal B^{n_1}_1,[h_1])$, $(\mathcal B^{n_2}_2,[h_2])$, respectively.  We consider its product $\mathcal X=\mathcal X_1\times \mathcal X_2$ with the product metric $g^+=g_1^+ + g_2^+$.  Then there exist defining functions $y_i$ on $\mathcal X_i$ so that
\begin{equation}\label{metric-product}
g^+_i=\frac{dy_i^2+h_{y_i}}{y_i^2},\quad h_{y_i}|_{y_i=0}=h_i, \quad i=1,2.
\end{equation}

From \cite{Mazzeo-Vasy:products} we have that for $\mathcal X_1$ and $\mathcal X_2$  the distinguished boundary can be identified with $\mathcal B_1\times \mathcal B_2$, while the Martin boundary consists of three pieces
$$(\partial \mathcal X_1\times \mathcal X_2)\cup (\partial \mathcal X_1\times\partial\mathcal X_2\times S_\mu)\cup
(\mathcal X_1\times \partial\mathcal X_2),$$
as in the model case \eqref{three-pieces}.

It is well known that the continuous spectrum of $-\Delta_{g_i}$ is the interval $\left[\frac{n^2}{4},\infty\right)$, but there could be a discrete number of points in $[0,\frac{n^2}{4})$. Let us assume that we are not in these exceptional cases for now.

The main result of this section is the following theorem:

\begin{thm}
 \label{thm:expansion-of-sol}
If $u$ is strongly $(\mu_1,\mu_2)$-harmonic on $\mathcal X$, then $u$ has an asymptotic behavior near the distinguished boundary $\mathcal B_1\times \mathcal B_2$ given by
\begin{equation}
\label{expansion-product}
u=F\, y_1^{\frac{n_1}{2}-\gamma_1}y_2^{\frac{n_2}{2}-\gamma_2}+
G\,y_1^{\frac{n_1}{2}-\gamma_1}y_2^{\frac{n_2}{2}+\gamma_2}
+H\, y_1^{\frac{n_1}{2}+\gamma_1}y_2^{\frac{n_2}{2}-\gamma_2}+
I\, y_1^{\frac{n_1}{2}+\gamma_1}y_2^{\frac{n_2}{2}+\gamma_2},
\end{equation}
where $F$, $G$, $H$, $I$ are $\mathcal C^\infty(\mathcal X_1\times \mathcal X_2)$ up to the boundary  if $f$ is $\mathcal C^\infty$.
\end{thm}

Note that the proof of this result in the particular case that $\mathcal X$ is exactly the product of two hyperbolic spaces  as in Section \ref{subsection:product-hyperbolic}  follows directly from \cite{KKMOOT} applied to a rank two symmetric space. But this proof would not work as it is for the general curved case, so we provide a new direct proof that does not use any tools from representation theory.

In any case, for the product of two hyperbolic spaces the Weyl group consists of four elements $\{e,\omega_1,\omega_2,\omega_1\omega_2\}         $, which are the possible reflections across two orthogonal axes, so it is very natural to expect the asymptotic expansion \eqref{expansion-product}.

Keeping the same notation as in the previous section, for simplicity we assume in the first two subsections that we are not in the exceptional value case. We start by constructing formally a
 strongly $(\mu_1,\mu_2)$-harmonic on $\mathcal X$ in Subsection \ref{formalconstruction} and identifying the scattering map in this case. Subsection \ref{rigorousconstruction} gives a rigorous proof of Theorem \ref{thm:expansion-of-sol}.

\subsubsection{Formal proof of Theorem \ref{thm:expansion-of-sol}}\label{formalconstruction}

First we give a formal (but constructive) solution to the system \eqref{system1}-\eqref{system2} for the product $\mathcal X^{N_1}\times \mathcal X^{N_2}$.

\begin{lemma}\label{lemma:formal-solution} Assume that $s_1>\frac{n_1}{2}$, $s_1\not\in \frac{n_1}{2}+\mathbb N$ and $s_2>\frac{n_2}{2}$, $s_2\not\in \frac{n_2}{2}+\mathbb N$.
Given $f\in\mathcal C^\infty(\mathcal B_1\times \mathcal B_2)$, there exists a formal solution $u$ in $\mathcal X_1 \times \mathcal X_2$ for the system
\begin{eqnarray}
\label{system1}&D_1 u:=-\Delta_1 u-s_1(n_1-s_1)u=0,\\
\label{system2}&D_2(u):=-\Delta_2 u-s_2(n_2-s_2)u=0,
\end{eqnarray}
with the asymptotic behavior $u=y_1^{n_1-s_1} y_2^{n_2-s_2}F$, $F\in\mathcal C^\infty$ up to the boundary, $F=f$ on $\mathcal B_1\times\mathcal B_2$.
\end{lemma}

\begin{proof}
The proof is a consequence of \cite{Graham-Zworski}, summarized in Lemma \ref{lemma-dim1} above, and we follow their notation.
We will construct $F$ as
$$F=\sum_{i=0}^\infty \sum_{j=0}^\infty f_{ij} y_1^{i} y_2^j, \quad\mbox{where}\quad f_{00}:=f.$$
Looking only at the first variable $x_1$, we set
$$f_{ij}:=p^{(1)}_{i,s_n1} f_{0j}, \quad j=0,1,\ldots,$$
where $p^{(1)}_{i,s_1}$ is the corresponding differential operator from \eqref{operator-p} in the variable $x_1$. Next, looking only at the second variable $x_2$, we set
$$f_{0j}:=p^{(2)}_{j,s_2} f.$$
Note that the operators $p^{(1)}_{i,s_1}$ and $p^{(2)}_{j,s_2}$ commute, so we have the identity
$$p^{(1)}_{i,s_1}p^{(2)}_{j,s_2} f=f_{ij}=p^{(2)}_{j,s_2}p^{(1)}_{i,s_1}f.$$
In particular, the first equality shows that $y_1^{n_1-s_1}F_{*j}$ for $F_{*j}:=\sum_{i} f_{ij} y_1^{i}$  is a formal solution to \eqref{system1} for every $j=0,1,\ldots$, while the second equality tells us that  $y_2^{n_2-s_2}F_{i*}$ for $F_{i*}:=\sum_{j} f_{ij} y_2^{j}$ is a formal solution to \eqref{system2}. Finally, by linearity, the function $u=y_1^{n_1-s_1} y_2^{n_2-s_2}F$ is a solution to both \eqref{system1} and \eqref{system2}.
This completes the proof of the Lemma.
\end{proof}

Notice that the proof above does not guarantee that the power series above converges. However,
as in \cite{Graham-Zworski}, using Borel's lemma is it possible to obtain $F$ such that in $\mathcal X_1\times \mathcal X_2$, $u:=y_1^{n_1-s_1}y_2^{n_2-s_2}F$ satisfies
\begin{eqnarray*}
\label{system1aprox}-\Delta_1 u-s_1(n_1-s_1)u=O(y_1^{\infty}),\\
\label{system2aprox}-\Delta_2 u-s_2(n_2-s_2)u=O(y_2^\infty).
\end{eqnarray*}
In the next subsection it will also suffice to conclude that the sum is taken up to a higher enough power such that $\phi_1:= -\Delta_1 u-s_1(n_1-s_1)u\in L^2(\mathcal X_1)$ and $\phi_2:= -\Delta_2 u-s_2(n_2-s_2)u\in L^2(\mathcal X_2)$.

\begin{lemma}\label{lemma-four-terms}
Given $F$ as in Lemma \ref{lemma:formal-solution}, there exist terms $G$, $H$, $I$ smooth up the whole boundary of $\mathcal X_1\times \mathcal X_2$ so that
\begin{equation*}
u=y_1^{n_1-s_1}y_2^{n_2-s_2}  F +y_2^{s_2} y_1^{n_1-s_1 }G+y_1^{s_1}y_2^{n_2-s_2}H+y_1^{s_1}y_2^{s_2} I.
\end{equation*}
 is a formal solution to the system \eqref{system1}-\eqref{system2}.
\end{lemma}

\begin{proof}
The main idea is to look separately at each equation in the system, keeping the other variable fixed, and construct a solution piece by piece.

In the first step, we write
$$u=y_2^{n_2-s_2} \left( y_1^{n_1-s_1}F+y_1^{s_1}H \right)+y_2^{s_2}\left( y_1^{n_1-s_1 }G+y_1^{s_1} I\right),$$
where $F$ is the approximate solution given in Lemma \ref{lemma:formal-solution}, and $H$, $G$, $I$ are to be found.

\textbf{Claim 1.} There exists an exact solution to \eqref{system1} of the form
$$y_1^{n_1-s_1} F+y_1^{s_1}H,$$
for some $$H=\sum_{i}\sum_j h_{ij} y_1^i y_2^j.$$

\noindent To see this claim, write $$F=\sum_i \sum_j f_{ij} y_2^jy_1^i=\sum_j F_{*j} y_2^{j},$$
for
$$F_{*j}:=\sum_{i} f_{ij} y_1^{i}. $$
Recall from the construction of Lemma \ref{lemma:formal-solution} that
\begin{equation}\label{sos1}
f_{ij}=p^{(1)}_{i,s_1} p^{(2)}_{j,s_2}f,\end{equation}
 so that $F_{*j}$ is a formal solution to \eqref{system1} for each fixed $j=0,1,\ldots$, where we are keeping
$x_2$ fixed. From Lemma \ref{lemma-dim1}, given $F_{*j}$,
there exists an exact solution to  \eqref{system1} of the form
$$y_1^{n_1-s_1} F_{*j}+y_1^{s_1}H_{*j},$$
for some
$$H_{*j}:=\sum_i h_{ij} y_1^i. $$
Moreover, formula \eqref{operator-q} gives the following characterization
\begin{equation}\label{sos3}h_{0j}=\mathcal S^{(1)}(s_1) f_{0j},\quad h_{ij}=q^{(1)}_{i,s_1} h_{0j},\end{equation}
for each $j=0,1,\ldots$. The claim follows by linearity.\\

\textbf{Claim 2.} $y_2^{n_2-s_2}H$ is a formal solution to \eqref{system2}.

\noindent This follows because
we can commute
\begin{equation*}
D_2(y_2^{n_2-s_2}H)=\sum_{i}  y_1^{i} D_2 \big(y_2^{n_2-s_2} \sum_j h_{ij}y_2^{j} \big )
= \sum_i y_1^{i}\, q^{(1)}_{i,s_1} \mathcal S^{(1)}(s_1) D_2\big(y_2^{n_2-s_2}\sum_j f_{0j} y_2^j\big),
\end{equation*}
and $y_2^{n_2-s_2}\sum_j f_{0j} y_2^j$ is a formal solution to \eqref{system2} from the proof of Lemma \ref{lemma:formal-solution}. The claim is proved.\\

Next, at the second step we interchange the role of $y_1$, $y_2$, writing
$$u=y_1^{n_1-s_1} \left( y_2^{n_2-s_2}F+y_2^{s_2}G \right)+y_1^{s_1}\left( y_2^{n_2-s_2 }H+y_2^{s_2} \tilde I\right).$$

\textbf{Claim 3.} There exists an exact solution to \eqref{system2} of the form
$$y_2^{n_2-s_2} F+y_2^{s_2}G,$$
for some $$G=\sum_{j}\sum_i g_{ij} y_1^i y_2^j.$$

\noindent The proof of this claim is analogous to the one of Claim 1. Just note that
\begin{equation}\label{sos2}g_{i0}=\mathcal S^{(2)}(s_2) f_{i0},\quad g_{ij}=q^{(2)}_{j,s_2} g_{i0}.\end{equation}

\textbf{Claim 4.} Given $G$ as in the previous claim, there exists an exact solution to \eqref{system1} of the form
$$y_1^{n_1-s_1} G+y_1^{s_1}I,$$
for some $$I=\sum_{j}\sum_i w_{ij} y_1^i y_2^j.$$

\noindent To show this, first we need to make sure that $G$ provides a good Dirichlet data for \eqref{system1}. But this is a consequence of the fact that, for each $j=0,1,\ldots$ fixed,
$$g_{ij}=p_{i,s_1}^{(1)}  \left(p_{0,s_2}^{(2)}q_{j,s_2}^{(2)}\mathcal S^{(2)}(s_2)f\right),$$
which follows by combining formulas \eqref{sos1} and \eqref{sos2}. The claim follows then by applying Lemma \ref{lemma-dim1}. Moreover, it gives that
\begin{equation}\label{sos4}w_{0j}=\mathcal S^{(1)}(s_1) g_{0j},\quad w_{ij}=q^{(1)}_{i,s_1} w_{0j}.\end{equation}

\textbf{Claim 5.} $y_1^{n_1-s_1}G$ is a formal solution to \eqref{system1}.

\noindent The proof is analogous to that of Claim 2.\\

\textbf{Claim 6.} Given $H$ as in Claim 1, there is an exact solution to \eqref{system2} of the form
$$y_2^{n_2-s_2} H+y_2^{s_2}\tilde I,$$
for some $$\tilde I=\sum_{i}\sum_j \tilde w_{ij} y_1^i y_2^j.$$

\noindent We follow the ideas from Claim 4. First, from \eqref{sos1} and \eqref{sos3} we have that
$$H=p^{(2)}_{j,s_2} \left( p_{0,s_1}^{(1)} q_{i,s_1}^{(1)}\mathcal S^{(1)}(s_1) f\right),$$
so it gives a good starting Dirichlet data for \eqref{system2}. Then
Lemma \ref{lemma-dim1} determines $\tilde I$ as stated in the claim. Moreover, for every  $i=0,1,\ldots$ fixed,
\begin{equation}\label{sos5} \tilde w_{i0}=\mathcal S^{(2)}(s_2) h_{i0},\quad\tilde w_{ij}=q^{(2)}_{j,s_2} \tilde w_{i0}.\end{equation}
This concludes the proof of the claim.\\

To finish the proof of Lemma \ref{lemma-four-terms} one needs to check that our choices are compatible and that they determine the same $I=\tilde I$.
But, from \eqref{sos4} and \eqref{sos2} one has
\begin{equation*}
w_{ij}=q^{(1)}_{i,s_1} \mathcal S^1(s_1) q^{(2)}_{j,s_2} \mathcal S^2(s_2) f,
\end{equation*}
and from \eqref{sos5} and \eqref{sos3},
$$\tilde w_{ij}=q^{(2)}_{j,s_2} \mathcal S^2(s_2) q^{(1)}_{i,s_1} \mathcal S^1(s_1) f.$$
Because the operators are applied to different variables $(x_1,y_1)$ and $(x_2,y_2)$, they commute and we automatically get $I=\tilde I$, as desired. Moreover, Claims 1 and 4 show that the $u$ we have constructed is a solution to \eqref{system1}, while Claims 3 and 6 yield that $u$ is a solution to \eqref{system2}, so we have constructed a suitable solution for the whole system.
\end{proof}

Note that the computations above are formal since they do not guarantee the convergence of the series. We will deal with this issue in the coming subsection.

\subsubsection{Proof of Theorem \ref{thm:expansion-of-sol} and further comments }\label{rigorousconstruction}

In this subsection we give a rigorous construction of the functions $G$, $H$ and $I$ of Lemma
\ref{lemma-four-terms}.

\begin{lemma}\label{lemma-four-terms2}
Given $F$ as in Remark \ref{approximate solution}, there exist terms $G$, $H$, $I$ smooth up the whole boundary of $\mathcal X_1\times \mathcal X_2$ so that
\begin{equation}\label{expansion-u}
u:=y_1^{n_1-s_1}y_2^{n_2-s_2}  F +y_2^{s_2} y_1^{n_1-s_1 }G+y_1^{s_1}y_2^{n_2-s_2}H+y_1^{s_1}y_2^{s_2} I.
\end{equation}
 is an exact solution to the system \eqref{system1}-\eqref{system2}.
\end{lemma}

\begin{proof}
As before, the main idea is to look separately at each equation in the system, keeping the other variable fixed, and construct a solution piece by piece. For $i=1,2$, let us denote by $\mathcal R_i$ the resolvent operator in $(\mathcal X^{N_i},g_i^+)$ given by Lemma \ref{lemma-dim1}. Note that $\mathcal R_i$ only acts on the coordinates $x_i$ of $\mathcal X^{N_i}$ and in particular $\mathcal R_1 \mathcal R_2 f(x_1,x_2)=\mathcal R_2 \mathcal R_1 f(x_1,x_2)$.

Let us define
\begin{align*}y_1^{s_1}H(x_1,y_1,x_2,y_2)=&\mathcal R_1 f,\\
y_2^{s_2}G(x_1,y_1,x_2,y_2)=&\mathcal R_2 f,\\
\hat{I}(x_1,y_1,x_2,y_2)=&\mathcal R_1(y_2^{s_2} G)=\mathcal R_1 \mathcal R_2 f=\mathcal R_2 \mathcal R_1 f=\mathcal R_2(y_1^{s_1} H),\\
y_1^{s_1}y_2^{s_2}I(x_1,y_1,x_2,y_2)=&\hat{I}.
\end{align*}
Note that the decays above are guaranteed by \cite{Mazzeo-Melrose} and arise from the decay of the resolvents $\mathcal R_1$ and $\mathcal R_2$.

From Lemma  \ref{lemma-dim1} we have that
$u_1= y_1^{n_1-s_1}F+y_1^{s_1} H$ is a solution to \eqref{system1}. Similarly,
$u_2=y_2^{n_2-s_2}F+y_2^{s_2} G$ is a solution to \eqref{system2} and
$u_3=  y_1^{n_1-s_1}G+y_1^{s_1} I$ is a solution to \eqref{system1}. However, since $R_1$ and
$R_2$ commute, we also have that $u_4=  y_2^{n_2-s_2}H+y_2^{s_2} I$ is a solution to \eqref{system2}.  Now, using the linearity,  it is direct to check that  $u=y_1^{n_1-s_1}y_2^{n_2-s_2}  F +y_2^{s_2} y_1^{n_1-s_1 }G+y_1^{s_1}y_2^{n_2-s_2}H+y_1^{s_1}y_2^{s_2} I$ is a solution to \eqref{system1}-\eqref{system2}.
\end{proof}

Note that the proof of Theorem \ref{thm:expansion-of-sol} follows from the previous Lemma.\\

\begin{remark}\label{remark-poles}
Looking at the modifications in Remark \ref{polosGZ} to treat the pole values, the same ideas will work if, for instance $s_1=\frac{n}{2}+k$, $k_1\in\mathbb N$. In particular, expansion \eqref{expansion-u} needs to be replaced by:
\begin{itemize}
\item If $s_2\not\in\frac{n}{2}+\mathbb N$,
$$u=y_1^{n_1-s_1}y_2^{n_2-s_2}  F +y_2^{s_2} y_1^{n_1-s_1 }G+y_1^{s_1}\log y_1y_2^{n_2-s_2}H+y_1^{s_1}\log y_1y_2^{s_2} I.$$
\item If $s_2=\frac{n}{2}+k_2$, $k_2\in\mathbb N$,
$$u=y_1^{n_1-s_1}y_2^{n_2-s_2}  F +y_2^{s_2} \log y_2y_1^{n_1-s_1 }G+y_1^{s_1}\log y_1y_2^{n_2-s_2}H+y_1^{s_1}\log y_1y_2^{s_2} \log y_2 I.$$
\end{itemize}
The scattering map is also well defined in this case (by using a residue formula).
\end{remark}

Our next objective is to deepen into the relation of our result to those of \cite{Mazzeo-Vasy:products}, at least in the product of two hyperbolic spaces $\mathcal X=\mathbb H^{N_1}\times\mathbb H^{N_2}$, through the use of Mellin transform. We will use the notation from Subsection \ref{subsection:product-hyperbolic}.

The results in \cite{Mazzeo-Vasy:products} concern the behavior of the resolvent $\mathcal R$ for the equation
\begin{equation}\label{equation-weak}
\Delta_{\mathcal X} v+\mu v=\phi.\end{equation}
The Mellin transform of a function $v:(0,\infty)\to \mathbb R$ is defined as
$$M[v](\zeta)=\int_0^\infty v(t) t^{\zeta-1}\,dt.$$
Take the Mellin transform of equation \eqref{equation-weak}
both in the variables $y_1$ and $y_2$ (denoted by $M_{12}$). By well known properties, we arrive to
$$[P_1(\zeta_1)+P_2(\zeta_2)] M_{12}[v]=M_{12}[\phi],$$
where $P_\alpha$ is the inditial polynomial for $\Delta_\alpha+s_\alpha(n_\alpha-s_\alpha)$, $\alpha=1,2$, i.e.,
$$P_\alpha(\zeta_\alpha)=\zeta_\alpha^2+n_\alpha \zeta_\alpha+s_\alpha(n_\alpha-s_\alpha),$$
and we have set $\mu=\mu_1+\mu_2$ with the usual conventions.
Let
$$P(\zeta_1,\zeta_2):=P_1(\zeta_1)+P_2(\zeta_2).$$
It is clear that this polynomial vanishes (only) on the circle
$$\left(\zeta_1+\tfrac{n_1}{2}\right)^2+\left(\zeta_2+\tfrac{n_2}{2}\right)^2=\gamma_1^2+\gamma_2^2=:|\gamma|^2,$$
so one needs to be careful when we invert the Mellin transform due to the presence of singularities
\begin{equation}\label{inverse-Mellin}v=\mathcal R(\mu)\phi:=\frac{1}{(2\pi i)^2} \int_{c_1-i\infty}^{c_1+i\infty} \int_{c_2-i\infty}^{c_2+i\infty} \frac{M_{12}[\phi]}{P(\zeta_1,\zeta_2)} y_1^{-\zeta_1} y_2^{-\zeta_2}\,d\zeta_1 \,d\zeta_2.  \end{equation}

One of the main results in \cite{Mazzeo-Vasy:products} is the calculation of the asymptotic behavior of the resolvent for \eqref{equation-weak} near the corner $\mathcal B_1\times \mathcal B_2=\{y_1=y_2=0\}$. Indeed, they show that
\begin{equation*}\label{resolventMV}\mathcal R(\mu)\phi \sim y_1^{n_1/2} y_2^{n_2/2} e^{\frac{|\gamma|}{\rho}},\end{equation*}
where $\rho^{-1}=\sqrt{\rho_1^{-2}+\rho_2^{-2}}$, $\rho_\alpha=-1/\log y_\alpha$, $\alpha=1,2$. Their proof can be recovered by using the stationary phase lemma to estimate the asymptotic behavior of the integral \eqref{inverse-Mellin} when $y_1,y_2\to 0$ (see, for instance, the book \cite{Bender-Orszag} for a standard reference on the stationary point lemma).

Let us see, at least formally, how this compares to our method. One just needs to observe that, if we take an approximate solution
$$u_0=y_1^{n_1-s_1}y_2^{n_2-s_2}  F +y_2^{s_2} y_1^{n_1-s_1 }G+y_1^{s_1}y_2^{n_2-s_2}H,$$
for equation \eqref{equation-weak},
we can find an exact solution of the form $u=u_0+v$, for
$$v=\mathcal R(\mu)\phi_0,\quad \phi_0=-(\Delta_{12}+\mu)(u_0).$$
We would like to show that such $v$ is the remaining term in the expansion \eqref{expansion-product}, i.e.,
\begin{equation}\label{equation30}v\sim y_1^{s_1} y_2^{s_2},\quad \text{as } y_1,y_2\to 0.\end{equation}

For this, define $D_\alpha:=\Delta_\alpha+\mu_\alpha$, $\alpha=1,2$, and $D_{12}=D_1+D_2$. Calculate
$$-\phi_1:=D_1(u_0)=y_2^{n_2-s_2} D_1\left(y_1^{n_1-s_1}F+y_1^{s_1}H\right)+y_2^{s_2}D_1\left(y_1^{n_1-s_1}G\right).$$
From Claim 1 in Lemma \ref{lemma-four-terms}, we know that the first term vanishes because we have constructed an exact solution to \eqref{system1}. The second term in the sum above vanishes in the variable $y_1$ up to infinite order thanks to Claim 5 in that lemma.
A similar calculation tells us that
$$-\phi_2:=D_2(u_0)=y_1^{n_1-s_1} D_2\left(y_2^{n_2-s_2}F+y_2^{s_2}G\right)+y_1^{s_1}D_2\left(y_2^{n_2-s_2}H\right)$$
vanishes up to infinite order in $y_2$, thanks to Claims 3 and 2.  We conclude that
$$-\phi=D_{12}u_0=-\phi_1-\phi_2=O(y_1^{\infty}y_2^{s_2})+O(y_1^{s_1}y_2^{\infty}).$$
Using the resolvent in the rank one case, since $\phi_1=O(y_1^{\infty}y_2^{s_2})$, there exists a solution $v_1$ for
$D_1(v_1)=\phi_1$ with the asymptotic behavior $v_1\sim y_1^{s_1}y_2^{s_2}$. Taking the Mellin transform, and looking at the first order term,
$$M_1[\phi_1]=P_1(\zeta_1) M_1[v_1]\sim P_1(\zeta_1)\delta_{y_1-s_1} y_2^{s_2}.$$
An analogous calculation gives that
$$M_2[\phi_2]=P_2(\zeta_2) M_2[v_2]\sim P_2(\zeta_2)\delta_{y_2-s_2} y_1^{s_1}.$$
In any case,
\begin{equation*}\label{Mphi}
M_{12}[\phi]=M_{12}[\phi_1]+M_{12}[\phi_2]\sim \left[P_1(\zeta_1)+P_2(\zeta_2)\right] \delta_{y_1-s_1} \delta_{y_2-s_2},\end{equation*}
so we can calculate $v$ as the inverse Mellin transform \eqref{inverse-Mellin}.
But the presence of the delta functions above yields \eqref{equation30}, as desired.

\subsection{The scattering map}

Equipped with Theorem \ref{thm:expansion-of-sol}, we are ready to define the scattering map with respect to the metric \eqref{metric-product} for the product of two conformally compact Einstein manifolds. As usual, we assume that we are not in the exceptional values.

Inspired by \eqref{scattering} in the symmetric space case, looking at the expansion \eqref{expansion-product}, it is natural to define
\begin{equation}\label{scattering-definition}
\mathcal S^{(\mu_1,\mu_2)} f:=\begin{pmatrix}
f & g_{00}\\
h_{00} & w_{00}
\end{pmatrix},\quad \text{for}\quad \mu_1\in\big(0,\left(\tfrac{n_1}{2}\right)^2\big),\, \mu_2\in\big(0,\left(\tfrac{n_2}{2}\right)^2\big),
\end{equation}
where $f$, $g_{00}$, $h_{00}$ and $w_{00}$ are, respectively, the boundary values of $F$, $G$, $H$ and $I$ on the distinguished boundary..
Because of our construction, it is easy to explicitly identify all the terms,
\begin{equation}\label{scattering-model}
\mathcal S^{(\mu_1,\mu_2)} f=\begin{pmatrix}
Id & \mathcal S^{(2)}(s_2)\\
\mathcal S^{(1)}(s_1) & \mathcal S^{(1)}(s_1)\,\mathcal S^{(2)}(s_2)
\end{pmatrix} f.
\end{equation}

Let us try to understand the behavior of this construction under conformal changes. However, in order to  mimmick the symmetric space setting of the product of two hyperbolic spaces, we just allow conformal changes in each variable separately. Thus, we perform a conformal change
$$\tilde h_\alpha=e^{2v_\alpha} h_\alpha, $$
for some $v_\alpha$ function on $\mathcal B_\alpha$.
Then one may a new defining function $\tilde y_\alpha$, $\alpha=1,2$, so that the metric is rewritten as
$$g^+_\alpha=\frac{d\tilde y_\alpha^2+\tilde h_{y_\alpha}}{\tilde y_\alpha^2},\quad h_{\tilde y_\alpha}|_{\tilde y_\alpha=0}=h_\alpha.$$
Moreover, the new defining function satisfies
$$\frac{\tilde y_\alpha}{y_\alpha}=e^{v_\alpha}+l.o.t. \text{ near }\{y_\alpha=0\}.$$
The scattering operator satisfies the following conformal property (componentwise):
\begin{equation*}
\tilde{\mathcal S}^{(\mu_1,\mu_2)}(f)=\begin{pmatrix}
e^{-(\frac{n_1}{2}-\gamma_1)w_1}e^{-(\frac{n_2}{2}-\gamma_2)w_2} \mathcal S_{11}f& e^{-(\frac{n_1}{2}-\gamma_1)w_1}e^{-(\frac{n_2}{2}+\gamma_2)w_2}\mathcal S_{12}f\\
e^{-(\frac{n_1}{2}+\gamma_1)w_1}e^{-(\frac{n_2}{2}-\gamma_2)w_2} \mathcal S_{21}f&
e^{-(\frac{n_1}{2}+\gamma_1)w_1}e^{-(\frac{n_2}{2}+\gamma_2)w_2}\mathcal S_{22}f
\end{pmatrix},
\end{equation*}
for the matrix
\begin{equation*}
\mathcal Sf:=\mathcal S^{(\mu_1,\mu_2)}\lp e^{(\frac{n_1}{2}-\gamma_1)w_1}e^{(\frac{n_2}{2}-\gamma_2)w_2} f\rp,
\end{equation*}
where we have denoted $\gamma_\alpha=\sqrt{\left(\frac{n_\alpha}{2}\right)^2-\mu_\alpha}$, $\alpha=1,2$.\\

We note that in the case of the product of two hyperbolic spaces the scattering matrix can be expressed in terms of the standard fractional Laplacian in $\mathbb{R}^n$. More precisely, in the expansion \eqref{expansion-product} we have

\begin{align*}
d_{\gamma_2} G|_{\{y_1=y_2=0\}}=&(-\Delta_{x_2})^{\gamma_2} f, \\
d_{\gamma_1} H|_{\{y_1=y_2=0\}}=&(-\Delta_{x_1})^{\gamma_1} f, \\
d_{\gamma_1}d_{\gamma_2} I|_{\{y_1=y_2=0\}}=&(-\Delta_{x_2})^{\gamma_2} (-\Delta_{x_1})^{\gamma_1}f,
\end{align*}
where $d_{\gamma_i}=2^{2\gamma_i}\frac{\Gamma(\gamma_i)}{\Gamma(-\gamma_i)}$, $i=1,2$.

We remark here that in \cite{Huang:thesis} the author studies the resolvent for \eqref{equation-weak} in the spirit of \cite{Mazzeo-Vasy:products}, and constructs a scattering-type operator in the weak $\mu$-harmonic function setting. This operator is a combination, depending on the angle (which is equivalent to the combination of $(\mu_1,\mu_2)$, $\mu_1+\mu_2=\mu$, of both $\mathcal S^{(1)}(s_1)$ and $\mathcal S^{(2)}(s_2)$. However, there is not a clear notion of conformality for this operator.

\subsection{The asymptotically product case}

Asymptotically product hyperbolic metrics were considered in great generality in
 \cite{Biquard-Mazzeo:products}, but there is not yet a uniform definition. The resolvent and the scattering operator for weakly harmonic functions were considered in \cite{Mazzeo-Vasy:products,Huang:thesis}. Here we would like to obtain the analogue to Theorem \ref{thm:expansion-of-sol} for strongly harmonic functions in order to define a scattering operator with good conformal properties. The main idea is to consider a perturbation of the product of two conformally compact Einstein manifolds $\mathcal X=\mathcal X_1\times \mathcal X_2$. In order to simplify our presentation, we just allow compact perturbations, but it is clear that our method would work for a more general perturbation as long as it has the right decay near the walls. Again, we stay away from the exceptional values.

\begin{thm}\label{thm:expansion2} Let $(\mathcal X_1,g_1)$, $(\mathcal X_1,g_1)$ be two conformally compact Einstein manifolds with conformal infinities $(\mathcal B_1,[h_1])$, $(\mathcal B_2,[h_2])$ and defining functions $y_1,y_2$, respectively.
Let $k$ be a compact perturbation for the metric $g_0=g_{1}+g_{2}$, with support away from all walls, and consider the metric on $\mathcal X:=\mathcal X_1\times \mathcal X_2$ given by $g^+=g_0+k$.

Then, given $f$ smooth on $\mathcal B_1\times\mathcal B_2$ and $\mu_1,\mu_2$, there exists a solution to
\begin{equation}-\Delta_{g^+} u-\mu u=0\quad \text{in }\mathcal X \label{full equation in product}\end{equation}
with the asymptotic expansion near the walls as in  \eqref{expansion-product}, for $\mu=\mu_1+\mu_2$.
\end{thm}

\begin{proof}
Let $f$ be a smooth function on $\mathcal B_1\times\mathcal B_2$ and fix $\mu_1,\mu_2$ as above.
Let $u_0$ be the solution to
$$-\Delta_{g_0} u-\mu u=0$$
 given by Lemma \ref{lemma-four-terms}. That is
 $$u_0= u:=y_1^{n_1-s_1}y_2^{n_2-s_2}  F +y_2^{s_2} y_1^{n_1-s_1 }G+y_1^{s_1}y_2^{n_2-s_2}H+y_1^{s_1}y_2^{s_2} I,$$
 where $\mu_i=s_i(n_i-s_i)$.

 We conclude the result by finding a function $v$ such that
 $u=u_0+v$ is a solution to \eqref{full equation in product} and that decays as $y_1^{s_1}y_2^{s_2}$. For this, we write
 $$-\Delta_{g^+} v-\mu v= A(x) ,\quad \text{in }\mathcal X,$$
where $A(x)=\Delta_{g^+} u_0+\mu u_0$. From the choice of $g^+$ we have that  $\textrm{supp }(A)\subset \textrm{ supp }(k)$, which is compact.

 The existence of $v$ will be proved by showing the existence of a resolvent between the correct function spaces that has a prescribed behavior.
To construct the resolvent we will follow  \cite{Mazzeo-Melrose}.

Let  $\mathcal R_1, \, \mathcal R_2$ be the functions defined in the proof of Lemma \ref{lemma-four-terms2}. That is, $\mathcal{R}_i$ is the resolvent in $\mathcal{X}_i$ of the operator $-\Delta_i-s_i(n_i-s_i)$. The mapping properties of $\mathcal{R}_i$ and the definition of the metric imply that for the space of smooth functions that vanish at all orders at infinity $\dot{\mathcal C}^\infty(\mathcal{X})$ the operator $\mathcal{R}_1\mathcal{R}_2:\dot{\mathcal C}^\infty(\mathcal{X})\to y_1^{s_1}y_2^{s_2}\mathcal C^\infty(\mathcal{X}) $ is well defined. Moreover, outside the support of $k$ we have that $\mathcal{R}_1\mathcal{R}_2=\mathcal{R}_2\mathcal{R}_1$.

Let $L(s)=-\Delta_{g^+}-s(n-s)$.
Then for any given $h\in \dot{\mathcal C}^\infty(\mathcal{X}) $ we have that
$$L(s) \mathcal{R}_1\mathcal{R}_2 -Id=-E(s),$$
 where $E(s):\dot{\mathcal C}^\infty(\mathcal{X}) \to \mathcal C_k^\infty(\mathcal{X})$. Here $\mathcal C_k^\infty(\mathcal{X})$ is the set of smooth functions supported in the support of $k$.
 Since  the set $\dot{\mathcal C}^\infty(\mathcal{X})$ is a dense subspace of $L^2(\mathcal{X})$,
 the construction also implies that  $E(s)$ is a $L^2(\mathcal{X})$ compact operator. Applying Fredholm theory we have that $[I-E(s)]$ is invertible and the inverse is meromorphic in $s$. Set $F(s)$ such that
 $$[I-E(s)]^{-1}=Id+E_1(s).$$

 We start by showing that $E_1$ has the correct mapping properties. More precisely $E_1(s):\dot{\mathcal C}^\infty(\mathcal{X}) \to \mathcal C_k^\infty(\mathcal{X}). $
 From the definition we have that
 $$E_1=E+E_1E=E+EE_1.$$

 The first equality implies that $E_1$ maps $\dot{\mathcal C}^\infty(\mathcal{X}) $ into $L^2(\mathcal{X})$, then the second equality implies $\dot{\mathcal C}^\infty(\mathcal{X}) $ is mapped into $\mathcal C^\infty_k (\mathcal{X}) $.

 Now we observe that the operator $\mathcal{R}: \dot{\mathcal C}^\infty(\mathcal{X}) \to y_1^{s_1}y_2^{s_2} \mathcal C^\infty(\mathcal{X}) $ given by
 $$\mathcal{R}= \mathcal{R}_1\mathcal{R}_2+ \mathcal{R}_1\mathcal{R}_2 E_1$$
 is a resolvent for $-\Delta_{g^+}-s(n-s)$ with the desired decay. Now we take $v=\mathcal{R}A$ and $u=u_0+v$ is the solution to \eqref{full equation in product} with the asymptotics of \eqref{expansion-product}.

 \end{proof}

 \begin{remark}
For every $\mu$-harmonic function in $(\mathcal{X}, g^+)$ a construction similar to the previous one can be done.
 In particular, this implies that the Martin boundary agrees with $\partial \mathcal{X}_1\times \partial\mathcal{X}_2 \times S_\mu$ as before.
 \end{remark}

Theorem \ref{thm:expansion2} allows to define the scattering operator in this setting with the same expression \eqref{scattering-definition}. However, as expected, in this general case we have lost the product structure \eqref{scattering-model} of the model, but it is is still a conformally covariant operator.\\

\noindent\textbf{Acknowledgements.}
The authors are grateful to Olivier Biquard,
  Sagun Chanillo, Matias Courdurier, Charles Fefferman, Rafe Mazzeo and Yiannis Sakellaridis
for many insightful discussions.

M. S\'aez was partially supported by Proyecto  Fondecyt  Regular 1150014 and
M.d.M. Gonz\'alez is supported by Spanish government grant MTM2014-52402-C3-1-P, and the BBVA foundation grant for  Investigadores y Creadores Culturales. Both authors would like to acknowledge the support of the  NSF grant DMS-1440140 while both authors were in residence at the Mathematical Sciences Research Institute in Berkeley, CA, during the Spring 2016 semester.




\begin{thebibliography}{10}

\bibitem{Anker-Orsted}
J.-P. Anker, B.~Orsted.
\newblock {\em Lie theory}, volume 229 of {\em Progress in Mathematics}.
\newblock Birkh\"auser Boston, Inc., Boston, MA, 2005.
\newblock Unitary representations and compactifications of symmetric spaces.

\bibitem{Bender-Orszag}
C. Bender, S. Orszag. {\em Advanced mathematical methods for scientists and engineers. I. Asymptotic methods and perturbation theory}. Springer-Verlag, New York, 1999.

\bibitem{Biquard-Mazzeo:parabolic}
O. Biquard, R. Mazzeo.
Parabolic geometries as conformal infinities of Einstein metrics.
{\em Arch. Math. (Brno)} 42 (2006), suppl., 85--104.

\bibitem{Biquard-Mazzeo:products} O. Biquard, R. Mazzeo. A nonlinear Poisson transform for Einstein metrics on product spaces. {\em J. Eur. Math. Soc.} 13, no. 5, 1423--1475,  2011.

\bibitem{Borel:book} A. Borel.
{\em Semisimple groups and Riemannian symmetric spaces}.
Texts and Readings in Mathematics, 16. Hindustan Book Agency, New Delhi, 1998.

\bibitem{Branson:sharp-inequalities}
T.~P. Branson.
\newblock Sharp inequalities, the functional determinant, and the complementary
  series.
\newblock {\em Trans. Amer. Math. Soc.}, 347(10):3671--3742, 1995.

\bibitem{Cabre-Serra} X. Cabre, J. Serra. An extension problem for sums of fractional Laplacians and $1-D$ symmetry of phase transitions. {\em Nonlinear Anal}. 137 (2016), 246--265.

\bibitem{Caffarelli-Silvestre}
L.~Caffarelli, L.~Silvestre.
\newblock An extension problem related to the fractional {L}aplacian.
\newblock {\em Comm. Partial Differential Equations}, 32(7-9):1245--1260, 2007.

\bibitem{Cap-Slovak}
A. C\v{a}p, J. Slov\'ak.
{\em Parabolic geometries. I.
Background and general theory}.
Mathematical Surveys and Monographs, 154. American Mathematical Society, Providence, RI, 2009.

\bibitem{Carron-Pedon}
G.~Carron and E.~Pedon.
\newblock On the differential form spectrum of hyperbolic manifolds.
\newblock {\em Ann. Sc. Norm. Super. Pisa Cl. Sci. (5)}, 3(4):705--747, 2004.

\bibitem{Chang-Gonzalez}
S.-Y. A.~Chang, M.~Gonz{\'a}lez.
\newblock Fractional {L}aplacian in conformal geometry.
\newblock {\em Adv. Math.}, 226(2):1410--1432, 2011.

\bibitem{Conrad}
K. Conrad.
\newblock Decomposing  $SL_2(\mathbb{R})$
\url{http://www.math.uconn.edu/~kconrad/blurbs/grouptheory/SL(2,R).pdf}

\bibitem{Frank-Gonzalez-Monticelli-Tan}
 R. Frank, M.d.M. Gonz\'alez, D. Monticelli, J. Tan.
Conformal fractional Laplacians on the Heisenberg group. {\em Advances in Mathematics}, 270 (2015) 97--137.

\bibitem{Freire} A. Freire.
On the {M}artin boundary of {R}iemannian products.
{\em J. Differential Geom.} 33, no. 1, 215--232, 1991.

\bibitem{Fulton-Harris}
W.~Fulton, J.~Harris.
\newblock {\em Representation theory}, volume 129 of {\em Graduate Texts in
  Mathematics}.
\newblock Springer-Verlag, New York, 1991.
\newblock A first course, Readings in Mathematics.

\bibitem{Furstenberg} H. Furstenberg.
A Poisson formula for semi-simple Lie groups.
{\em Ann. of Math}. (2) 77, 335--386,  1963.

\bibitem{Gangolli:spherical}
R.~Gangolli.
\newblock Spherical functions on semisimple {L}ie groups.
\newblock In {\em Symmetric spaces ({S}hort {C}ourses, {W}ashington {U}niv.,
  {S}t. {L}ouis, {M}o., 1969--1970)}, pages 41--92. Pure and Appl. Math., Vol.
  8. Dekker, New York, 1972.

\bibitem{Gindikin-Karpelevi}
S.~G. Gindikin and F.~I. Karpelevi{\v{c}}.
\newblock Plancherel measure for symmetric {R}iemannian spaces of non-positive
  curvature.
\newblock {\em Dokl. Akad. Nauk SSSR}, 145:252--255, 1962.

\bibitem{Gonzalez:survey}
M.d.M. Gonz\'alez.
Recent progress on the fractional Laplacian in conformal geometry.
Preprint.

\bibitem{Graham-Zworski} C.R. Graham, M. Zworski.
Scattering matrix in conformal geometry.
{\em Invent. Math. 152} (2003), no. 1, 89--118.

\bibitem{GJT}
Y.~Guivarc'h, L.~ Ji, J.C.~ Taylor.
\newblock Compactifications of symmetric spaces.
\newblock {\em Progress in Mathematics, } 156. Birkhuser Boston, Inc., Boston, MA, 1998.

\bibitem{Giulini-Woess} S. Giulini, W. Woess.
\newblock The Martin compactification of the Cartesian product of two hyperbolic spaces.
\newblock {\em J. Reine Angew. Math}. 444:17--28, 1993.

\bibitem{Harish-Chandra:sphericalI}
Harish-Chandra.
\newblock Spherical functions on a semisimple {L}ie group. {I}.
\newblock {\em Amer. J. Math.}, 80:241--310, 1958.

\bibitem{Harish-Chandra:sphericalII}
Harish-Chandra.
\newblock Spherical functions on a semisimple {L}ie group. {II}.
\newblock {\em Amer. J. Math.}, 80:553--613, 1958.

\bibitem{Helgason:duality}
S.~Helgason.
\newblock A duality for symmetric spaces with applications to group
  representations.
\newblock {\em Advances in Math.}, 5:1--154 (1970), 1970.

\bibitem{Helgason:jewel}
S.~Helgason.
\newblock Harish-{C}handra's {$c$}-function. {A} mathematical jewel.
\newblock In {\em Noncompact {L}ie groups and some of their applications ({S}an
  {A}ntonio, {TX}, 1993)}, volume 429 of {\em NATO Adv. Sci. Inst. Ser. C Math.
  Phys. Sci.}, pages 55--67. Kluwer Acad. Publ., Dordrecht, 1994.

\bibitem{Helgason:libro3}
S.~Helgason.
\newblock {\em Groups and geometric analysis}, volume~83 of {\em Mathematical
  Surveys and Monographs}.
\newblock American Mathematical Society, Providence, RI, 2000.
\newblock Integral geometry, invariant differential operators, and spherical
  functions, Corrected reprint of the 1984 original.

\bibitem{Helgason:libro2}
S.~Helgason.
\newblock {\em Differential geometry, {L}ie groups, and symmetric spaces},
  volume~34 of {\em Graduate Studies in Mathematics}.
\newblock American Mathematical Society, Providence, RI, 2001.
\newblock Corrected reprint of the 1978 original.

\bibitem{Helgason:libro}
S.~Helgason.
\newblock {\em Geometric analysis on symmetric spaces}, volume~39 of {\em
  Mathematical Surveys and Monographs}.
\newblock American Mathematical Society, Providence, RI, second edition, 2008.

\bibitem{KoriyaneiHelgasonboundary} S. Helgason, A. Kor\'anyi. A Fatou-type theorem for harmonic functions on symmetric spaces.  {\em Bull. Amer. Math. Soc.}  74  (1968) 258--263.


\bibitem{Huang:thesis}
H.-C. Huang.
{\em The Resolvent and Scattering Operator on Near-Product-Hyperbolic Spaces}. Phd Thesis. Stanford University. 2006.

\bibitem{Joshi-SaBarreto} M. Joshi, A. S\'a Barreto. Inverse scattering on asymptotically hyperbolic manifolds.
{\em Acta Mathematica}, 184, 41--86, (2000).



 \bibitem{Baum-Juhl} H. Baum, A. Juhl. {\em Conformal differential geometry. $Q$-curvature and conformal holonomy}. Oberwolfach Seminars, 40. Birkhäuser Verlag, Basel, 2010.

\bibitem{Karpelevic}
F.~Karpelevic.
\newblock The geometry of geodesics and the eigenfunctions of the Beltrami-Laplace operator on symmetric spaces.
\newblock {\em Trans. Moscow Math. Soc}. 14:51--199, 1965.


\bibitem{KKMOOT}
M.~Kashiwara, A.~Kowata, K.~Minemura, K.~Okamoto, T.~{\=O}shima, and M.~Tanaka.
\newblock Eigenfunctions of invariant differential operators on a symmetric
  space.
\newblock {\em Ann. of Math. (2)}, 107(1):1--39, 1978.

\bibitem{Kashiwara-Oshima:regular-singularities}
M.~Kashiwara and T.~{\=O}shima.
\newblock Systems of differential equations with regular singularities and
  their boundary value problems.
\newblock {\em Ann. of Math. (2)}, 106(1):145--200, 1977.

\bibitem{Knapp}
A.~W. Knapp.
\newblock {\em Lie groups beyond an introduction}, volume 140 of {\em Progress
  in Mathematics}.
\newblock Birkh\"auser Boston Inc., Boston, MA, second edition, 2002.


\bibitem{Knappboundary}  A. Knapp.
Fatou's theorem for symmetric spaces. I.
{\em Ann. of Math.} (2)  88  (1968) 106--127.


\bibitem{Koriyaniboundary1} A. Kor\'anyi. Boundary behavior of Poisson integrals on symmetric spaces. {\em Trans. Amer. Math. Soc.}  140  (1969) 393--409.

\bibitem{Koriyaniboundary2} A. Kor\'anyi.
Poisson integrals and boundary components of symmetric spaces. {\em Invent. Math.}  34  (1976), no. 1, 19--35.


\bibitem{Koranyi:survey} A. Kor\'anyi.
A survey of harmonic functions on symmetric spaces. Harmonic analysis in Euclidean spaces (Proc. Sympos. Pure Math., Williams Coll., Williamstown, Mass., 1978), Part 1, pp. 323344,
Proc. Sympos. Pure Math., XXXV, Part, Amer. Math. Soc., Providence, R.I., 1979.

\bibitem{Martin}
R. Martin.
\newblock Minimal positive harmonic functions.
\newblock {\em Trans. Amer. Math. Soc.} 49:137--172, 1941.

\bibitem{Mazzeo-Melrose} R. Mazzeo, R. Melrose. Meromorphic extension of the resolvent on complete spaces with asymptotically constant negative curvature. {\em J. Funct. Anal.}, 75(2):260--310, 1987.

\bibitem{Mazzeo-Vasy:products}
 R. Mazzeo, A. Vasy. Resolvents and Martin boundaries of product spaces. {\em Geom. Funct. Anal.} 12, no. 5, 1018--1079,  2002.

\bibitem{Mazzeo-Vasy:SL3} R. Mazzeo, A. Vasy. Analytic continuation of the resolvent of the Laplacian on $SL(3)/SO(3)$. {\em Amer. J. Math.} 126, no. 4, 821--844, 2004.

  \bibitem{Mazzeo-Vasy:SL3bis}
  R. Mazzeo, A. Vasy. Scattering theory on $SL(3)/SO(3)$: connections with quantum 3-body scattering. {\em Proc. Lond. Math. Soc.} (3) 94, no. 3, 545--593, 2007.

\bibitem{Mazzeo-Vasy:noncompact} R. Mazzeo, A. Vasy. Analytic continuation of the resolvent of the Laplacian on symmetric spaces of noncompact type. {\em J. Funct. Anal}. 228 (2005), no. 2, 311368.

\bibitem{Michelson} H. L. Michelson. Fatou theorems for eigenfunctions of the invariant differential operators on symmetric spaces. {\em Trans. Amer. Math. Soc.}  177  (1973), 257--274.

\bibitem{Oshima:polos} T. Oshima. Boundary value problems for systems of linear partial differential equations with regular singularities. {\em Group representations and systems of differential equations (Tokyo, 1982)}, 391--432, Adv. Stud. Pure Math., 4, North-Holland, Amsterdam, 1984.

\bibitem{Oshima:commuting}
T. Oshima.
Commuting differential operators with regular singularities. {\em Algebraic analysis of differential equations from microlocal analysis to exponential asymptotics},  195–224, Springer, Tokyo, 2008.

\bibitem{Oshima-Sekiguchi}
T. \={O}shima, J. Sekiguchi.
Eigenspaces of invariant differential operators on an affine symmetric space.
Invent. Math. 57, no. 1, 1--81, 1980.

\bibitem{Sato-Kawai-Kashiwara}
M. Sato, T. Kawai, M. Kashiwara.
\newblock    Microfunctions and pseudo-differential equations,
 {\em Hyperfunctions and pseudo-differential equations ({P}roc.
              {C}onf., {K}atata, 1971; dedicated to the memory of {A}ndr\'e
              {M}artineau)},
\newblock      pp 265--529. Lecture Notes in Math., Vol. 287,
 {\it Springer, Berlin}, 1973.


\bibitem{Schlichtkrull} H. Schlichtkrull.
{\em Hyperfunctions and harmonic analysis on symmetric spaces}.
Progress in Mathematics, 49. Birkh\"auser Boston, Inc., Boston, MA, 1984.

\bibitem{Schlichtkrull:paper} H. Schlichtkrull. On the boundary behaviour of generalized Poisson integrals on symmetric spaces. {\em Trans. Amer. Math. Soc.}  290  (1985),  no. 1, 273--280.

\bibitem{Sjoren} P. Sj\"ogren. A Fatou theorem for eigenfunctions of the Laplace-Beltrami operator in a symmetric space. {\em Duke Math.} J.  51  (1984),  no. 1, 47--56.

 \bibitem{Sugiura} M. Sugiura.
{\em Unitary representations and harmonic analysis. An introduction}. North-Holland Mathematical Library, 44. North-Holland Publishing Co., Amsterdam; Kodansha, Ltd., Tokyo, 1990.


\end{thebibliography}

\end{document}